\documentclass{scrartcl}
\usepackage[utf8]{inputenc}

\usepackage[english]{babel}
\usepackage{amsmath,amssymb}
\usepackage{enumitem} 
\usepackage{mathtools}
\usepackage{amsthm}
\usepackage{thmtools}
\usepackage{graphics}
\usepackage{multicol}
\usepackage{hyperref}
\usepackage{mathabx,fdsymbol}
\usepackage[many]{tcolorbox}    
\usepackage{todonotes}
\newtcolorbox{mybox}[1]{%
    tikznode boxed title,
    hbox,
    enhanced,
    arc=0mm,
    interior style={white},
    attach boxed title to top center= {yshift=-\tcboxedtitleheight/2},
    colbacktitle=white,coltitle=black,
    size = small,
    boxed title style={size=fbox,colframe=white,boxrule=0pt},
    title={#1}}

\declaretheorem[numberwithin=section, name=Lemma]{lemma}

\declaretheorem[name=Theorem,sibling=lemma]{theorem}
\declaretheorem[name=Corollary,sibling=lemma]{corollary}
\declaretheorem[name=Conjecture,sibling=lemma]{conjecture}
\declaretheorem[style=definition, name=Example,sibling=lemma]{example}
\declaretheorem[style=definition,name=Definition,sibling=lemma]{definition}
\declaretheorem[style=definition, name=Remark,sibling=lemma]{remark}
\declaretheorem[style=definition, name=Question,sibling=lemma]{question}

\newcommand\xqed[1]{%
  \leavevmode\unskip\penalty9999 \hbox{}\nobreak\hfill
  \quad\hbox{#1}}
\newcommand\eoe{\xqed{$\triangle$}}

\newcommand{\N}{\mathbb{N}}
\newcommand{\lcm}{\operatorname{lcm}}

\newcommand{\A}{{\mathbb A}}
\newcommand{\B}{{\mathbb B}}
\newcommand{\C}{{\mathbb C}}
\newcommand{\D}{{\mathbb D}}
\newcommand{\F}{{\mathbb F}}
\newcommand{\Downset}{{\Gamma}}
\newcommand{\Cyc}[1]{{\mathbb C}_{#1}}
\newcommand{\KThree}{{\mathbb K}_3}
\newcommand{\shiftTuple}[2]{{#1}_{#2}}
\newcommand{\Pol}{\ensuremath{\operatorname{Pol}}}
\newcommand{\Hom}{\ensuremath{\operatorname{Hom}}}

\newcommand{\PL}{\ensuremath{\operatorname{PCL}}}

\newcommand{\FD}{\ensuremath{\mathcal F_D(\omega)}}

\newcommand{\PPPoset}{{\mathfrak P}_{\operatorname{fin}}}
\newcommand{\SDPoset}{{\mathfrak P}_{\operatorname{UC}}}
\newcommand{\PCPoset}{{\mathfrak P}_{\operatorname{UPC}}}

\newcommand{\Flat}{\operatorname{rad}}

\newcommand{\MPCL}{\mathfrak M_{\PL}}
\newcommand{\MCL}{\mathfrak M_{\operatorname{CL}}}
\newcommand{\mCL}{M_{\operatorname{CL}}}
\newcommand{\mPCL}{M_{\PL}}

\usetikzlibrary{arrows,arrows.meta}
\usetikzlibrary{bending,calc}
\newcommand{\cycle}[2]{
\def \n {#1}
\def \radius {{#1*0.035+0.15}}
\def \margin {3.8/(#1*0.035+0.15)} 
\node[scale={0.8+#1/10}] at #2 {#1};

\foreach \s in {1,...,\n}
{
  \draw #2+({360/\n * (\s - 1)}:\radius) circle (1pt);
  \draw[>=stealth',arrows=-{>[bend]}] #2+({360/\n * (\s - 1)+\margin}:\radius) 
    arc ({360/\n * (\s - 1)+\margin}:{360/\n * (\s)-\margin}:\radius);
}}

\newcommand{\cycleEmpty}[3]{
\def \n {#1}
\def \radius {{#1*0.035+0.15}}
\def \margin {3.8/(#1*0.035+0.15)} 

\foreach \s in {1,...,\n}
{
  \draw #2+({90+360/\n * (\s - 1)}:\radius) circle (1pt);
}
\foreach \s in #3
{
    \draw[>=stealth',arrows=-{>[bend]}] #2+({90+360/\n * (\s - 1)+\margin}:\radius) 
    arc ({90+360/\n * (\s - 1)+\margin}:{90+360/\n * (\s)-\margin}:\radius);
}
}

\title{Smooth digraphs modulo primitive positive constructability and cyclic loop conditions}
\author{M.~Bodirsky, F.~Starke\thanks{The second author is supported by DFG Graduiertenkolleg 1763 (QuantLA).}, A.~Vucaj\thanks{The first and the third author have received funding from the European Research Council (ERC Grant
Agreement no. 681988, CSP-Infinity).}\\\small{Institut für Algebra, TU Dresden, 01062 Dresden, Germany}}

\begin{document}

\usetikzlibrary{arrows,calc}

\maketitle

\begin{abstract}
Finite smooth digraphs, that is, finite directed graphs without sources and sinks, can be partially ordered via pp-constructability. 
We give a complete description of this poset and, in particular, we prove that it is a distributive lattice. Moreover, we show that in order to separate two smooth digraphs in our poset
it suffices to show that the polymorphism clone of one of the digraphs satisfies a prime cyclic loop condition that is not satisfied by the polymorphism clone of the other. 
Furthermore, we prove that the poset of cyclic loop conditions ordered by their strength for clones is a distributive lattice, too.
\end{abstract}



\tableofcontents

\section{Introduction}

We consider a poset which is closely related to \emph{Constraint Satisfaction Problems} (CSPs). 
In 1999 Feder and Vardi conjectured that the class of CSPs over finite structures admits a dichotomy, i.e., every such problem is either polynomial-time tractable or NP-complete. Recently, Zhuk and Bulatov independently proved this conjecture~\cite{ZhukFVConjecture,BulatovFVConjecture}.
This dichotomy has an algebraic counterpart: the structures with NP-complete CSP are precisely those that \emph{pp-construct} $\tikz{
\node at (-30:0.2) [circle, fill, scale=0.3] (0) {};
\node at (90:0.2) [circle, fill, scale=0.3] (1) {};
\node at (210:0.2) [circle, fill, scale=0.3] (2) {};
\path 
    (0) edge (1)
    (1) edge (2)
    (2) edge (0);
}$, the complete graph on three vertices $\KThree$ (unless $\text{P}= \text{NP}$). It turns out that pp-constructability is a quasi-order on the class of all finite structures. There are log-space reductions between the CSPs of comparable structures~\cite{wonderland}.
Hence, understanding the arising poset (called \emph{pp-constructability poset}) can lead to a better understanding of the precise computational complexity of CSPs within P.

The same poset arises in universal algebra in two other ways. A finite structure $\A$ pp-constructs a finite structure $\B$ if and only if there is a
\emph{minor-preserving map} from the polymorphism clone of $\A$ to the polymorphism clone of $\B$~\cite{wonderland}; for definitions, see
Section~\ref{Sec:Preliminaries}.
A third way to describe this poset is a Birkhoff-like approach, extending the concept of a variety in universal algebra by a variety that is not only closed under homomorphic images, subalgebras, and products, but also closed under taking so-called \emph{reflections} of algebras.
We do not need this perspective in the present article and refer the reader to~\cite{wonderland}.

The present article is a step at the beginning of the journey to understand the pp-construct\-ability
poset on all finite structures. This may be a ground for a finer classification of finite domain CSPs than the P/NP-complete dichotomy of Bulatov and Zhuk. A complete description of the subposet arising from two-element structures is given in~\cite{albert}.

The subposet arising from undirected graphs is just a three-element chain~\cite{BulatovHColoring}: 
\[[\tikz{
\node at (-30:0.2) [circle, fill, scale=0.3] (0) {};
\node at (90:0.2) [circle, fill, scale=0.3] (1) {};
\node at (210:0.2) [circle, fill, scale=0.3] (2) {};
\path 
    (0) edge (1)
    (1) edge (2)
    (2) edge (0);
}] < [\tikz{
\node at (0.2,0) [circle, fill, scale=0.3] (1) {};
\node at (-0.1,0) [circle, fill, scale=0.3] (2) {};
\path 
    (1) edge (2);
}] < [\tikz{
\node at (-30:0.2) [circle, fill, scale=0.3] (0) {};
}].\]

However, studying the whole poset appears to be very difficult, since even seemingly simple cases, like for example directed graphs, are not well understood. In fact, in an early attempt to close in on the Feder-Vardi conjecture, researchers tried (hard) to classify which oriented trees have an NP-complete CSP, only managing to get results for particular classes of oriented trees \cite{BartoBulin, BartoKozikMarotiNiven, BulinColoring, HellTreeHomomorphisms, HellDualityTree}.
However, already before the result of Bulatov and Zhuk it has been proved that the P/NP-complete dichotomy holds for a particular family of directed graphs, i.e., for every finite directed graph such that every vertex has an incoming and an outgoing edge~\cite{BartoKozikNiven}.
Such digraphs are known in the literature as \emph{finite smooth digraphs}. From a result of Barto, Kozik, and Niven~\cite{BartoKozikNiven} it follows that every finite smooth digraph whose CSP is not NP-complete can be represented in the poset by a finite disjoint union of directed cycles. 

In this article, we restrict our attention to finite smooth digraphs ordered by pp-constructability
and give a complete description of the corresponding poset.
In particular, it turns out that this poset is even a distributive lattice and that it suffices to consider finite disjoint unions of directed cycles of square-free lengths.
It is known that for any two structures that do not pp-construct each other there is a \emph{height 1 (strong Mal'cev) condition} that is satisfied by polymorphisms of one of the structures but not by polymorphisms of the other~\cite{wonderland}. The present article shows that for finite smooth digraphs this height 1 condition can be chosen to be a so-called \emph{cyclic loop condition}. An example for such a condition is 
\[f(x_0,x_1,y_0,y_1,y_2)\approx f(x_1,x_0,y_1,y_2,y_0).\]
Cyclic loop conditions are also related to finite disjoint unions of cycles in a more straightforward way. For example the above condition can be linked to the disjoint union of a cycle of length two and a cycle
of length three.

\subsection*{Outline}
In Sections~\ref{Sec:Preliminaries} and~\ref{sec:smoothDigraphs} we introduce basic notation, in particular pp-constructions, cyclic loop conditions, and disjoint unions of cycles.
To warm up, in Section \ref{sec:primefun}, we prove some of the results of this paper for the special case of disjoint unions of prime cycles and prime cyclic loop conditions.  
In Section~\ref{sec:conditions} we focus only on cyclic loop conditions. We give a description of the implication order on cyclic loop conditions (Theorem~\ref{thm:implicationSingleClc}) and on sets of cyclic loop conditions (Corollary~\ref{thm:implicationClc}). Part of this description is proved later as it depends on the interplay of cyclic loop conditions and finite disjoint unions of directed cycles. 
We show that every cyclic loop condition is equivalent to a set of prime cyclic loop conditions (Theorem~\ref{thm:clcAreSetsOfPclc}). Finally, we show that the poset of sets of prime cyclic loop conditions ordered by their strength for clones, and hence the poset of cyclic loop conditions ordered by their strength for clones, has a comprehensible description (Corollary~\ref{cor:characterizationMCLAndmCL}). 
In Section~\ref{sec:structures} we  first characterize the relation $\models$ on finite disjoint unions of directed cycles and cyclic loop conditions (Lemma~\ref{lem:duofcSatisfyClc}). We use this characterization to show the missing
part of the proof in Section~\ref{sec:conditions} mentioned above. Then we prove, using results from Section~\ref{sec:conditions}, that the pp-constructability type of a finite disjoint union of directed cycles is determined by the finite set of prime cyclic loop conditions that it does not satisfy (Lemma~\ref{thm:PPvsLoop}). Furthermore, we show that every reverse-implication closed finite set of prime cyclic loop conditions is realised by some  finite disjoint union of directed cycles (Lemma~\ref{lem:surjectivety}). These two statements give a complete description of the poset of  finite disjoint union of directed cycles ordered by pp-constructability. Together with the result from Barto, Kozik, and Niven we obtain a description of the poset of finite smooth digraphs ordered by pp-constructability  (Corollary~\ref{thm:classificationPoset}).
In Section~\ref{sec:lattice} we show that this poset is even a distributive lattice.

\section{Preliminaries}\label{Sec:Preliminaries}
In this section we present formal definitions of notions mentioned in the introduction.  

\subsection*{Notation}
\begin{itemize}
    \item For $n\in\N^+\!$, we define $[n]\coloneqq\{1,\dots,n\}$.
    \item By $\operatorname{Im}(f)$ we denote the image of $f$.
    \item By $\lcm$ and $\gcd$ we denote the least common multiple and the greatest common divisor, respectively.
    \item For a tuple $\boldsymbol{a} = (a_1,\dots,a_n)$ and a map $\sigma \colon [m]\to [n]$, we
denote the tuple $(a_{\sigma(1)},\dots,a_{\sigma(m)})$ by $\shiftTuple{\boldsymbol{a}}{\sigma}$.
    \item By $k\equiv_a \ell$ we denote $k=\ell\pmod a$
    \item By $\A\hookrightarrow \B$ we denote that $\A$ embeds into $\B$.
\end{itemize}

\subsection{The pp-constructability poset}

Let $\A=(A,(R^\A)_{R\in\tau})$ and $\B=(B,(R^\B)_{R\in\tau})$ be structures with the same relational signature $\tau$. A map
$h\colon A \to B$ is a \emph{homomorphism} from $\A$ to $\B$ if it preserves all relations, i.e., for all $R\in \tau$:
\[\text{if } (a_1,\dots,a_n)\in R^\A, \text{ then } (h(a_1),\dots,h(a_n))\in R^\B.\]
We write $\A\to\B$ if there exists a homomorphism from $\A$ to $\B$. 
If $\A\to\B$ and $\B\to\A$, then we say that $\A$ and $\B$ are \emph{homomorphically equivalent}.

Let $\mathbb A$ be a relational structure and $\phi(x_1,\dots,x_n)$ be a \emph{primitive positive} (pp-) formula, i.e., a first order formula using only existential quantification and conjunctions of atomic formulas. 
Then the relation 
\begin{equation*}
\{(a_1,\dots,a_n)\mid \mathbb A \vDash \phi(a_1,\dots,a_n)\}.
\end{equation*}
is said to be \emph{pp-definable} in $\mathbb A$.
We say that $\mathbb B$ is a \emph{pp-power} of $\A$ if it is isomorphic to a structure with domain $A^n$, for some $n\in \N^+$\!, whose relations are pp-definable in $\A$ (a $k$-ary relation on $A^n$ is regarded as a $kn$-ary relation on $A$).

Combining the notions of homomorphic equivalence and pp-power we obtain the following definition from~\cite{wonderland}.
\begin{definition}
We say that $\A$ \emph{pp-constructs} $\B$, denoted by $\A\leq\B$, if $\B$ is homomorphically equivalent to a pp-power of $\A$.
\end{definition}
Since pp-constructability is a reflexive and transitive relation on the class of relational structures~\cite{wonderland}, the use of the symbol $\leq$ is justified. 
Note that the quasi-order $\leq$ naturally induces the equivalence relation 
\[\A \equiv \B \text{ if and only if }\B \leq \A \leq \B.\]
The equivalence classes of $\equiv$ are called \emph{pp-constructability types}. We denote by $[\A]$ the pp-constructability type of a structure $\A$.
\begin{definition}
We name the poset
\[\PPPoset\coloneqq(\{[\A]\mid \A\ \text{a finite relational structure}\},\leq)\] the \emph{pp-constructability poset}.
\end{definition} 
Observe that $[\tikz{
\clip(-0.15,-0.053) rectangle (0.15,0.35);
\node[circle, fill, scale=0.3] at (0,0) (0) {};
\path (0) edge[out=120, in=60, looseness=20,->,>=stealth'] (0);
}]$, the pp-constructability type  of the loop graph, is the top element of $\PPPoset$. Later we will see that there is also a bottom element and that it is  $[\tikz{
\node at (-30:0.2) [circle, fill, scale=0.3] (0) {};
\node at (90:0.2) [circle, fill, scale=0.3] (1) {};
\node at (210:0.2) [circle, fill, scale=0.3] (2) {};
\path 
    (0) edge (1)
    (1) edge (2)
    (2) edge (0);
}]$.

\subsection{Height 1 identities}\label{h1}
As already mentioned, pp-constructability can be characterized algebraically; this characterization will provide the main tool to prove that a structure cannot pp-construct another structure. First, we define the basic notion of \emph{polymorphism}.
Let $\Hom(\A,\B)$ denote the following set:
\[\Hom(\A,\B)\coloneqq\{f\mid f \text{ a homomorphism from } \A\text{ to } \B\}.\]
For $n\geq 1$, we denote by $\A^n$ the structure
with same signature $\tau$ as $\A$ whose domain is $A^n$ such that for any $k$-ary $R\in \tau$, a tuple $(\boldsymbol{a}_1,\dots,\boldsymbol{a}_k)$ of $n$-tuples is contained in $R^{\A^n}$ if and only if it is contained in $R^\A$
componentwise, i.e., $(a_{1j},\dots,a_{kj})\in R^\A$ for all $1\leq j\leq n$.

\begin{definition}
For a relational structure $\A$, a \emph{polymorphism of $\A$} is a function that is an element of $\Hom(\A^n,\A)$, for some $n\in \N^+$\!.
Moreover, we denote by $\Pol(\A)$ the \emph{polymorphism clone} of $\A$, i.e., the set of all polymorphisms of $\A$.
\end{definition}

\begin{definition}
Let $\sigma \colon [m]\to [n]$ and $f\colon A^m \to A$ be functions. We define the function $f_\sigma\colon A^n\to A$ by the rule
\[f_\sigma(\boldsymbol{a}) \coloneqq f(\shiftTuple{\boldsymbol{a}} {\sigma}).\] 
Any function of the form $f_\sigma$, for some map $\sigma \colon [m]\to [n]$, is called a \emph{minor} of $f$.
\end{definition}

We extend the definitions of $\boldsymbol{a}_\sigma$ and $f_\sigma$ to $\sigma\colon I\to J$, $f\colon A^I\to A$, and $\boldsymbol{a}\in A^J$, where $I$ and $J$ are arbitrary sets, in the obvious way.
Let  $I$ be a finite set and $f\colon \A^I\to\A$ a homomorphism. Then for every bijection $\sigma\colon [|I|] \to I$ the function $f_\sigma$ is a polymorphism of $\A$.
Throughout this article we treat all homomorphisms $f\colon \A^I\to\A$, where $I$ is a finite set, as elements of $\Pol(\A)$ by identifying $f$ with $f_\sigma$ for a bijection $\sigma$ (in all such situations the particular choice of $\sigma$ does not matter). 

\begin{definition}
Let $\A$ and $\B$ be structures. An arity-preserving map $\lambda\colon \Pol(\B)\to\Pol(\A)$ is \emph{minor-preserving} if for all $f\colon B^m\to B$ in $\Pol(\B)$ and $\sigma\colon [m]\to [n]$ 
we have
\begin{equation*}
    \lambda(f_\sigma)=(\lambda f)_\sigma.
\end{equation*}
\end{definition}
 We write $\Pol(\B)\stackrel{\text{minor}}{\to}\Pol(\A)$ to denote that there is a minor-preserving map from $\Pol(\B)$ to $\Pol(\A)$.
The next theorem, restated from~\cite{wonderland}, shows that the concepts presented so far, pp-constructions and minor-preserving maps, give rise to the same poset.
\begin{theorem}[Theorem 1.3 in~\cite{wonderland}]\label{thm:ppvsminor}
Let $\A$ and $\B$ be finite structures. Then
\begin{align*}
    \B\leq\A&&\text{if and only if}&&\Pol(\B)\stackrel{\mathrm{minor}}{\to}\Pol(\A).
\end{align*}
\end{theorem}

\begin{definition}
Let $\sigma\colon [n]\to [r]$ and $\tau\colon [m]\to [r]$ be functions. A \emph{height 1 identity} is an expression of the form:
\begin{equation*}
    \forall x_1,\dots,x_r( f(x_{\sigma(1)},\dots,x_{\sigma(n)}) \approx g(x_{\tau(1)},\dots,x_{\tau(m)})).
\end{equation*}
\end{definition}
Usually, we write $f(x_{\sigma(1)},\dots,x_{\sigma(n)}) \approx g(x_{\tau(1)},\dots,x_{\tau(m)})$ omitting the universal quantification, or even $f_\sigma\approx g_\tau$ for brevity.
A finite set of height 1 identities is called \emph{height 1 condition}.
A set of functions $F$ \emph{satisfies} a set of height 1 identities $\Sigma$, denoted $F\models\Sigma$, if there is a map $\tilde\cdot$ assigning to each function symbol occurring in $\Sigma$ a function in $F$ such that for all $f_\sigma\approx g_\tau\in\Sigma$ we have $\tilde f_\sigma= \tilde g_\tau$. Note that we make a distinction between the symbol $\approx$ and $=$ to emphasize the difference between an identity in a first-order formula and an equality of two specific objects
A set of functions $\mathcal C$ is called a \emph{clone} if it contains the projections and is closed under composition, i.e., for every $n$-ary $f\in\mathcal C$ and all $m$-ary $g_1,\dots,g_n\in\mathcal C$ we have that $f(g_1,\dots,g_n)\in\mathcal C$. Note that every polymorphism clone is a clone.

\begin{definition}\label{def:orderOnConditions}
Let $\Sigma$ and $\Gamma$ be sets of height 1 identities. We say that $\Sigma$ \emph{implies} $\Gamma$ (and that $\Sigma$ \emph{is stronger then} $\Gamma$), denoted $\Sigma\Rightarrow\Gamma$, if 
\[\mathcal C \models \Sigma \text{ implies } \mathcal C \models \Gamma \text{ for all clones $\mathcal C$.} \] 
If $\Sigma\Rightarrow\Gamma$ and $\Gamma\Rightarrow\Sigma$, then we say that $\Sigma$ and $\Gamma$ are \emph{equivalent}, denoted $\Sigma\Leftrightarrow\Gamma$. We define 
\[[\Sigma]\coloneqq\{\Sigma'\mid \Sigma'\text{ a set of height 1 identities, }\Sigma\Leftrightarrow\Sigma'\}.\]
\end{definition}  
We say that a height 1 condition is \emph{trivial} if it is satisfied by projections on a set $A$ that
contains at least two elements, or, alternatively, if it is implied by any height 1 condition.
We extend the definition of $\models$ and $\Rightarrow$ to single functions and single height 1 identities in the obvious way. Hence, if $f$ is a function, $\Sigma$ is a height 1 identity, and $\Gamma$ is a set of height 1 identities, then we can write $f\models\Sigma$ instead of $\{f\}\models\{\Sigma\}$ and $\Gamma\Rightarrow\Sigma$ instead of $\Gamma\Rightarrow\{\Sigma\}$.

Observe that if $\lambda\colon\Pol(\B)\to\Pol(\A)$ is minor-preserving and $f_\sigma= g_\tau$, then $\lambda(f)_\sigma=\lambda(g)_\tau$. It follows that minor-preserving maps preserve height 1 conditions.
A simple compactness argument shows the following corollary.
\begin{corollary}\label{cor:wond}
Let $\A$ and $\B$ be finite structures. Then
\begin{align*}
    \B\leq\A&&\text{iff}&&\text{$\Pol(\B)\models\Sigma$ implies $\Pol(\A)\models\Sigma$ for all height 1 conditions $\Sigma$}.
\end{align*}
\end{corollary}
In general, showing that there is no minor-preserving map from $\Pol(\B)$ to $\Pol(\A)$ is a rather complicated task.
However, a recent result by Barto, Bul\'in, Krokhin, and Opr\v{s}al provides a concrete height 1 condition to check~\cite{jakubPCSP}.
We show that, for smooth digraphs $\mathbb G$, $\mathbb H$, whenever $\mathbb G\nleq\mathbb H$ there is a single height 1 identity with only one function symbol witnessing this. Height 1 identities of this form have been studied in the literature and are called \emph{loop conditions}~\cite{olsak-loop, olsak-strong}. 
 
\begin{definition}
Let $\sigma,\tau\colon [m] \to [n]$ be maps. A \emph{loop condition} is a height 1 identity of the form
\begin{equation*}\label{Sigma}
    f_\sigma \approx f_\tau.
\end{equation*}
\end{definition}
To any loop condition $\Sigma$ we can assign a digraph in a natural way. 
\begin{definition}
Let $\sigma,\tau\colon [m] \to [n]$ be maps and let $\Sigma$ be the loop condition, given by the identity $f_\sigma \approx f_\tau$. We define the digraph $\mathbb G_\Sigma\coloneqq
([n],\{(\sigma(i),\tau(i))\mid i\in [m])$. 
\end{definition}

\begin{example}\label{ex:G_Sigma} Some loop conditions and the corresponding digraphs.
\begin{itemize}
    \item Let $\Sigma_S$ be the loop condition $f(x,y,x,z,y,z)\approx f(y,x,z,x,z,y)$. Then $\mathbb G_{\Sigma_S}$ is isomorphic to $\KThree$.
    \item Let $\Sigma_3$ be the loop condition $f(x,y,z) \approx f(y,z,x)$. Then $\mathbb G_{\Sigma_3}$ is isomorphic to a directed cycle of length 3. \eoe
\end{itemize}
\end{example}
Observe that, for every digraph $\mathbb G$, all loop conditions $\Sigma$ such that $\mathbb G_{\Sigma}\simeq \mathbb G$ are equivalent. For convenience, we will from now on allow any finite set in place of $[m]$ and $[n]$. This allows us to construct from a graph a concrete loop condition.
\begin{definition}\label{def:SigmaG}
Let $\mathbb G=(V,E)$ be a digraph. We define the loop condition $\Sigma_{\mathbb G}\coloneqq (f_\sigma\approx f_\tau)$, where $\sigma,\tau\colon E\to V$ with $\sigma(u,v)=u$ and $\tau(u,v)=v$.
\end{definition} 
Observe that $\mathbb G_{\Sigma_\mathbb G}=\mathbb G$.
The name loop condition is justified by the following observation. If $\mathbb G$ is a finite graph such that $\Pol(\mathbb G)$ satisfies
$\Sigma$ and $\mathbb G_\Sigma \to \mathbb G$, then $\mathbb G$ has a loop. 
Consequently, if $\mathbb G$ does not have a loop, then $\Pol(\mathbb G)\not\models\Sigma_{\mathbb G}$. 
If $\mathbb G_\Sigma$ itself has a loop, then there is an $i$ with $\sigma(i)=\tau(i)$ and a structure $\A$ satisfies $\Sigma$ with the projection $\pi_i\colon\boldsymbol{a}\mapsto a_i$ and therefore $\Sigma$ is trivial. 
If $\mathbb G_\Sigma$ is a disjoint union of directed cycles, then we say that $\Sigma$ is a \emph{cyclic loop condition}. For instance, the identity $\Sigma_3$ in Example~\ref{ex:G_Sigma} is a cyclic loop condition.

\subsection{Free structures}\label{sec:freestructures}
Here we present another characterization of pp-constructability. This section can be safely skipped without compromising the understanding of the rest of the paper; its aim is just to
put our proof method for Lemma~\ref{thm:classificationPoset} into a larger context. The definition of free structure which we are going to adopt in this article was presented in~\cite{jakubPCSP}.

\begin{definition}\label{def:freestructure}
Let $\A$ be a finite relational structure on the set $A=[n]$, and $\mathcal C$ a clone (not necessarily related to $\A$). The \emph{free structure of} $\mathcal C$ \emph{generated by} $\A$ is a relational structure $\F_{\mathcal C}(\A)$ with the same signature as $\A$. Its universe $F_{\mathcal C}(A)$ consists of all $n$-ary operations in $\mathcal C$. For any relation of $\A$, say $R^{\A}=\{\boldsymbol{r}_1,\dots,\boldsymbol{r}_m\}\subseteq A^k$, the relation $R^{\F_{\mathcal C}(\A)}$ is defined as the set of all $k$-tuples $(f_1,\dots,f_k)\in F_{\mathcal C}(A)$ such that there exists an $m$-ary operation $g\in\mathcal C$ that satisfies
\[f_j(x_1,\dots,x_n) = g(x_{{r}_{1j}},\dots,x_{{r}_{mj}}) \text{ for each } j\in[k].\]
\end{definition}

The following theorem links the notion of free structure to the characterization of pp-constructability presented in Theorem~\ref{thm:ppvsminor}.

\begin{theorem}[Theorem 4.12 in~\cite{jakubPCSP}]\label{thm:freestructure}
Let $\A$ and $\B$ be finite relational structures. Then
\begin{align*}
    \B\leq\A&&\text{if and only if}&&\Pol(\B)\stackrel{\mathrm{minor}}{\to}\Pol(\A)&&\text{if and only if}&& \F_{\Pol(\B)}(\A)\to\A .
\end{align*}
\end{theorem}

\section{Smooth digraphs}\label{sec:smoothDigraphs}
A \emph{smooth digraph} is a directed graph $\mathbb G$ such that every vertex has an incoming and an outgoing edge.
Barto, Kozik, and Niven~\cite{BartoKozikNiven} showed  the following dichotomy for smooth digraphs.

\begin{theorem}\label{thm:BartoKozikNiven}
Let $\mathbb G$ be a finite smooth digraph. Then either $\mathbb G$ pp-constructs $\KThree$ or it is homomorphically equivalent to a finite disjoint union of directed cycles.
\end{theorem}

It is known that $\KThree$ pp-constructs every finite structure (see, e.g.,~\cite{Bodirsky-HDR}). Hence, $[\KThree]$ is the bottom element of $\PPPoset$. From the previous theorem we know that every non-minimal element in $\PPPoset$, represented by a smooth digraph, contains a finite disjoint union of directed cycles. 
Let us denote the subposet of $\PPPoset$ consisting of the pp-constructability types of finite disjoint unions of directed cycles by $\SDPoset$.
In Section~\ref{sec:structures} we will provide a
characterization of the pp-constructability order on disjoint unions of directed cycles.
As we are always considering digraphs, we will usually drop the word \emph{directed}.

To any finite set $C\subset\N^+\!$ we associate a finite disjoint union of cycles $\C=(V,E)$ defined by 
\begin{align*}
    V \coloneqq \{(a,k) \mid a\in C, k \in \mathbb Z_a\} && \text{and} && E \coloneqq \{((a,k),(a,k+_a 1)) \mid a\in C, k \in \mathbb Z_a\}, 
\end{align*} where by $+_a$ we denote the addition modulo $a$.
For the sake of notation we will from now on write $+$ instead of $+_a$; it will be clear from the context to which addition we are referring to.
For any $a_1,\dots,a_n\in\N^+\!$ we write $\Cyc{a_1,\dots,a_n}$ for the finite disjoint union of cycles associated to the set $\{a_1,\dots,a_n\}$. Note that, for $a\in\N^+$\!, the structure $\Cyc a$ is a directed cycle of length $a$.
To any finite disjoint union of cycles $\C$ we associate the set $C\coloneqq\{a\mid \Cyc{a}\hookrightarrow \C\}$. 

We warn the reader that previously $C$ denoted the underlying set of the structure $\C$,
but from now on $\C$ itself will denote the structure as well as the underlying set and $C$
is the set defined above. We hope that this will not lead to any confusion. Also beware
that the finite disjoint union of cycles $\D$ associated to the set associated to a finite disjoint union of cycles $\C$ is
not necessarily isomorphic to $\C$. The structure $\C$ could have multiple copies of the same
cycle whereas $\D$ may not. However, $\C$ and $\D$ are homomorphically equivalent and thus
represent the same element in $\SDPoset$.

We define for every $k\in\N^+$ the pp-formula
\[x\stackrel{k}\to y\coloneqq \exists y_1,\dots,y_{k-1}( E(x,y_1)\wedge E(y_1,y_2)\wedge\dots\wedge E(y_{k-1},z)),\]
which we will often use in pp-constructions.
For a disjoint union of cycles $\C$, $u,v\in \C$, and $k\in\N^+\!$ we have that $\C\models u\stackrel{k}\to v$ if there is a  directed  path of length $k$ from $u$ to $v$. If $\C$ is clear from the context we abbreviate $\C\models u\stackrel{k}\to v$ by $ u\stackrel{k}\to v$. 
Note that for fixed $u$ and $k$ there is exactly one $v \in \C$ such that $ u\stackrel{k}\to v$; we denote this element $v$ by $u+k$.
The \emph{$k$-th relational power of $\C$} is  the digraph $(\C,\{(u,u+k)\mid u\in\C\})$. 
We denote the map $\C\to\C$, $u\mapsto u+1$ by $\sigma_{\C}$. Note that $\sigma_{\C} \in \operatorname{Aut}(\C)$.

Observe that the cyclic loop condition $\Sigma_{\C}$, introduced in Definition~\ref{def:SigmaG}, is $f\approx f_{\sigma_{\C}}$ and that the edge relation of $\C$ is $\{(u,\sigma_{\C}(u))\mid u\in\C\}$. If $\C$ is the finite disjoint union of cycles associated to the set $C$, then we write $\sigma_{\! C}$ and $\Sigma_{C}$ instead of $\sigma_{\C}$ and $\Sigma_{\C}$, respectively. Note that $\sigma_{\! C}(a,k) = (a,k+1)$ for any $a\in C$ and $k\in\{0,\dots,a-1\}$.

We want to remark that any finite power $\mathbb C^n$ of a disjoint union of cycles $\mathbb C$ is again a disjoint union of cycles. Hence, for an element $\boldsymbol{t}\in\mathbb C^n$, $k\in\N^+\!$, we have that $\boldsymbol{t}+k$ is already defined, furthermore \[\boldsymbol{t}+k=(t_1+k,\dots,t_n+k).\]

For $a,c\in\N^+\!$ define $a\dotdiv c\coloneqq\frac{a}{\gcd(a,c)}$. Note that $a\dotdiv c$ is always a natural number.
The choice of the symbol $\dotdiv$ is meant to emphasize that $a\dotdiv c$ is the numerator of the fraction $a\div c$
in reduced form. Roughly speaking, the operation $\dotdiv$ should be understood as ``divide as much as you can". 
The operation $\dotdiv$ has the following useful properties.
\begin{lemma}\label{lem:dotdiv}
For all $a,b,c\in\N^+\!$ we have
\begin{enumerate}
    \item $a\dotdiv (a\dotdiv c)=\gcd(a,c)$,
    \item $(a\dotdiv b)\dotdiv c=a\dotdiv (b\cdot c)$,
    \item $\gcd(a\dotdiv c,c\dotdiv a)=1$, and
    \item $a \dotdiv c=1$ if and only if $a$ divides $c$. 
\end{enumerate}
\end{lemma}
\begin{proof}
Simply applying the definitions we obtain:
\begin{align*}
    a\dotdiv (a\dotdiv c)=\frac{a}{\gcd(a,a\dotdiv c)}=\frac{a}{\gcd\left(a,\frac{a}{\gcd(a,c)}\right)}
    =\frac{a}{\frac{a}{\gcd(a,c)}}
    =\gcd(a,c)
\end{align*}
The reader can verify the other statements.
\end{proof}
For a finite disjoint union of cycles $\C$ and $c\in\N^+\!$, we let $\C\dotdiv c$ denote the finite disjoint union of cycles associated to the set $C\dotdiv c\coloneqq\{a\dotdiv c\mid a\in C\}$.

\begin{lemma}\label{lem:relPowEqualsCdotdivc}
Let $\C$ be a finite disjoint union of cycles and $c\in\N^+\!$. The $c$-th relational power of $\C$ is homomorphically equivalent to $\C\dotdiv c$.
\end{lemma}
\begin{proof}
Note that it suffices to verify the claim for cycles.
Let $\C=\Cyc a$. Then $\C\dotdiv c$ consists of one cycle of length $a\dotdiv c$ and the  $c$-th relational power of $\C$ consists of $\frac{a}{a\dotdiv c}=\gcd(a,c)$ many cycles of length $a\dotdiv c$. Hence they are homomorphically equivalent.
\end{proof}

\begin{example}\label{ex.division}
The structure $\Cyc 6$ pp-constructs $\Cyc 3=\Cyc6\dotdiv2$. To see this, consider the first pp-power of $\Cyc 6$ given by the pp-formula:
\[ \Phi_E(x,y)\coloneqq x\stackrel{2}{\to}y.\]
We obtain a structure that consists of two disjoint copies of $\Cyc 3$, which is homomorphically equivalent to $\Cyc 3$; see Figure~\ref{fig:6pp3}. 
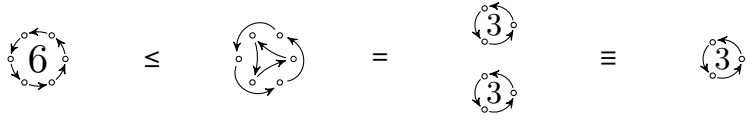
\begin{figure}
    \centering
    \begin{tikzpicture}
    \cycle{6}{(-3,0)}
    
    \def \n {3}
\def \radius {{2*3*0.035+0.15}}

\foreach \s in {1,...,\n}
{
  \draw ({360/\n * (\s - 1)}:\radius) circle (1pt);
  \draw ({360/\n * (\s - 1)+180/\n}:\radius) circle (1pt);
  \draw ({360/\n * (\s - 1)}:\radius) edge[>=stealth',arrows=-{>[bend]},bend left=20,shorten >=1mm,shorten <=1mm] ({360/\n * (\s)}:\radius);
  \draw ({360/\n * (\s - 1)+180/\n}:\radius) edge[>=stealth',arrows=-{>[bend]},bend right = 80,looseness=1.7,shorten >=1mm,shorten <=1mm] ({360/\n * (\s)+180/\n}:\radius);
}

    \cycle{3}{(3,0.45)}
    \cycle{3}{(3,-0.45)}
    
    \cycle{3}{(6,0)}
    
    \node at (-1.5,0) {$\leq$};
    \node at (1.5,0) {$=$};
    \node at (4.5,0) {$\equiv$};
    
    \end{tikzpicture}
    \caption{The structure $\Cyc6$ pp-constructs $\Cyc3$.}
    \label{fig:6pp3}
\end{figure}

Note that $\Cyc 6$ pp-constructs $\Cyc 2$ with the formula $\Phi_E(x,y)\coloneqq x\stackrel{3}{\to}y$.
In general, the first pp-power of a disjoint union of cycles $\C$ given by the formula $x\stackrel{c}{\to}y$ is the $c$-th relational power of $\C$, which is, by Lemma~\ref{lem:relPowEqualsCdotdivc}, homomorphically equivalent to $\C\dotdiv c$. Therefore, $\C\leq \C\dotdiv c$.
\eoe
\end{example}

\section{Disjoint unions of prime cycles and prime cyclic loop conditions}\label{sec:primefun}
Before we characterize the posets of finite disjoint union of directed cycles ordered by pp-constructability, $\SDPoset$, and cyclic loop conditions ordered by strength (cf. Definition~\ref{def:orderOnConditions}), in Sections \ref{sec:conditions} and \ref{sec:structures}, we study the following subposets.

\begin{definition}
A cyclic loop condition $\Sigma_P$ is called a \emph{prime cyclic loop condition} if $P$ is a set of primes. 
A disjoint unions of cycles $\mathbb P$ is called a \emph{disjoint unions of prime cycles} if $P$ is a set of primes.
We define the posets $\mPCL$ and $\PCPoset$ as 
\begin{align*}
    \mPCL&\coloneqq (\{[\Sigma] \mid \Sigma \text{ a prime cyclic loop condition}\},\Leftarrow),
    \\
    \PCPoset&\coloneqq (\{[\mathbb P]\mid \mathbb P\text{ a disjoint union of prime cycles}\},\leq).
\end{align*}

\end{definition} 

This section has two goals.
Firstly, we want to understand $\mPCL$ and $\PCPoset$ as these posets will be used to describe $\SDPoset$.
Secondly, this section gently introduces the reader to techniques used in Sections  \ref{sec:conditions} and \ref{sec:structures}  without many of the difficulties that come with the more general case.

To gain a better understanding of the objects we are working with, we reproduce a well-known fact about cycles and cyclic loop conditions, presented in Lemma~\ref{lem:pcSatisfyPclc}. We start with a simple example.
\begin{example}\label{exa:C3notmodelsS3}
Consider the digraph $\Cyc 3$. Observe that $\Pol(\Cyc 3) \models \Sigma_2$ as witnessed by the polymorphism
\[f(x,y) = 2\cdot(x + y) \pmod 3.\]
On the other hand, $\Pol(\Cyc 3) \not\models \Sigma_3$. Assume that $f$ is a polymorphism of $\Cyc 3$ satisfying $\Sigma_3$, then
\[
\begin{tikzpicture}
\node at (0,0.9) {$f((3,0),(3,1),(3,2))$};
\node at (0,0) {$f((3,1),(3,2),(3,0))$};
\node at (2.0,0.9) {$=a$};
\node at (2.0,0) {$=a$};
\path[<-,>=stealth'] 
    (-0.9,0.25) edge (-0.9,0.6)
    (0.15,0.25) edge (0.15,0.6)
    (1.2,0.25) edge (1.2,0.6)
    (2.15,0.25) edge (2.15,0.6)
    ;
\end{tikzpicture}\]
and  $(a,a)$ is a loop, a contradiction.
\eoe
\end{example}

\begin{lemma}\label{lem:pcSatisfyPclc}
Let $p,q$ be primes. Then $\Pol(\Cyc{q})\models\Sigma_p$ if and only if $p\neq q$.
\end{lemma}
\begin{proof}
If $p\neq q$, then there is an $n\in \N^+\!$ such that $p\cdot n\equiv_q1$. The map
\[f(x_1,\dots,x_p) = n\cdot(x_1 +\ldots + x_p) \pmod q\]
is a polymorphism of $\Cyc q$ satisfying $\Sigma_p$.

Assume that $f$ is a polymorphism of $\Cyc q$ satisfying $\Sigma_p$, then
\[f((q,0),\dots,(q,q-2),(q,q-1)) = a = f((q,1),\dots,(q,q-1),(q,0))\]
and $(a,a)$ is a loop, a contradiction.
\end{proof}
From Corollary~\ref{cor:wond} it is easy to see that the digraphs $\Cyc 2, \Cyc 3, \Cyc 5,\dots$ represent an infinite antichain in $\PCPoset$ and that the conditions $\Sigma_2,\Sigma_3,\Sigma_5,\dots$ represent an infinite antichain in $\mPCL$. 
In order to describe these two posets we generalize Lemma~\ref{lem:pcSatisfyPclc} to disjoint unions of prime cycles in Lemma~\ref{lem:duofpcSatisfyPclc}. 
First, we present an example that (hopefully) helps to better understand the polymorphisms of disjoint unions of cycles. 

\begin{example}
Let us examine the binary polymorphisms of $\Cyc{2,3}$ (see Figure~\ref{fig:23poly}). 
 
\begin{figure}
    \centering
    \begin{tikzpicture}
    \cycle{2}{(0,-0.4)}
    \cycle{2}{(0,0.4)}
    
    \cycle{3}{(2,0)}
    \cycle{3}{(2,1)}
    \cycle{3}{(2,-1)}
    
    \cycle{6}{(1,-0.5)}
    \cycle{6}{(1,0.5)}
    
    \path (3.5,0) edge[->,>=stealth'] (5.5,0);
    
    \cycle{2}{(7,-0.2)}
    \cycle{3}{(7.8,0.3)}
    \end{tikzpicture}
    \caption{Just a nice image to keep you motivated to understand the set $\Hom(\Cyc{2,3}^2,\Cyc{2,3})$, the binary polymorphisms of $\C_{2,3}$ are homomorphisms from the left digraph to $\C_{2,3}$. }
    \label{fig:23poly}
\end{figure}
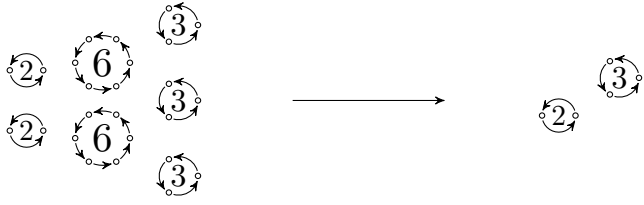 
Every element in $\Hom(\Cyc{2,3}^2,\Cyc{2,3})$ is build from homomorphisms of the connected components into $\Cyc{2,3}$. 
Hence, $\left|\Hom(\Cyc{2,3}^2,\Cyc{2,3})\right| = 2^2\cdot  3^3\cdot 5^2$.
More generally, let $\C$ be a finite disjoint union of cycles and $n\in \N^+$\!. Let $\textbf G$ be the subgroup of the symmetric group (even automorphism group) on $\C^n$ generated by $+1\colon\boldsymbol t\mapsto(\boldsymbol t+1)$.   
For every  orbit  of $\C^n$ under $\textbf G$ pick  a representative and denote the set of representatives by $T$. Note that every connected component of $\C^n$ contains exactly one element of $T$. Let $f\colon T\to\C$ be such that for every $\boldsymbol t=((a_1,k_1),\dots,(a_n,k_n))\in T$ we have that $f(\boldsymbol t)$ lies in a cycle whose length divides $\ell_{\boldsymbol t}\coloneqq\lcm(a_1,\dots,a_n)$. Since every $\boldsymbol t\in T$ lies in a cycle of length $\ell_{\boldsymbol t}$ we have that $f$ can be uniquely extended to a polymorphism $\tilde f$ of $\C$. Furthermore, 
\[\tilde f((+1)^d(\boldsymbol t))=f(\boldsymbol t)+d\]
for all $\boldsymbol t\in T$ and $d\in\N$. Note that all $n$-ary polymorphisms of $\C$ can be constructed this way.
\eoe
\end{example}

\begin{lemma}\label{lem:duofpcSatisfyPclc}
Let $\mathbb Q$ be a disjoint union of prime cycles and let $p$ be a prime. Then $\Pol(\mathbb Q)\models \Sigma_p$ if and only if $p\notin Q$.
\end{lemma}
\begin{proof}
We assume without loss of generality that $\mathbb Q$ is the disjoint union of cycles associated to $Q$.

($\Rightarrow$) 
Let $f$ be a polymorphism of $\mathbb Q$ satisfying $\Sigma_p$. Assume $p\in Q$. Then 
\[f((p,0),\dots,(p,p-2),(p,p-1)) = a = f((p,1),\dots,(p,p-1),(p,0))\]
and $(a,a)$ is a loop, a contradiction.

($\Leftarrow$)
For this direction we construct a polymorphism $f\colon \mathbb Q^{\Cyc p} \to \mathbb Q$ of $\mathbb Q$ satisfying $\Sigma_p$.
Let $\textbf G$ be the subgroup of the symmetric group on $\mathbb Q^{\Cyc p}$ generated by $\sigma_{\! p}\colon \boldsymbol t \mapsto \shiftTuple{\boldsymbol t }{\sigma_{\! p}}$ and $+1\colon \boldsymbol t\mapsto (\boldsymbol t+1)$. Recall that 
\begin{align*}
    (\shiftTuple{\boldsymbol t} {\sigma_{\! p}})_{(p,k)}=\boldsymbol t_{(p,k+1)},&& +1=\sigma_{ (\mathbb Q^{\Cyc p})}, && ((+1)(\boldsymbol t))_{(p,k)}=\boldsymbol t_{(p,k)}+1.
\end{align*}
Hence, $\sigma_{\! p}\circ (+1)=(+1)\circ \sigma_{\! p}$, $\textbf G$ is commutative, and every element of $\textbf G$ is of the form $\sigma_{\! p}^c\circ (+1)^d$ for some $c, d \in\N$.
For every orbit of $\mathbb Q^{\Cyc p}$ under $\textbf G$ pick a representative and denote the set of representatives by $T$. Let $\boldsymbol t\in T$ and $(q,k)=\boldsymbol t_{(p,0)}$.
Note that the orbit of $\boldsymbol t$ is a disjoint union of cycles of a fixed length $n$ and that $q$ divides $n$. Define $f$ on the orbit of $\boldsymbol t$ as 
\[f\left((\sigma_{\! p}^c\circ (+1)^d)(\boldsymbol t)\right)
=f\left(\shiftTuple {(\boldsymbol t+d)}{\sigma_{\! p}^c}\right)
\coloneqq (q,d)\text{ for all $c,d$.}\] 
To show that $f$ is well defined on the orbit of $\boldsymbol t$ it suffices to prove that $\shiftTuple{(\boldsymbol t+d)}{\sigma_{\! p}^c}=\shiftTuple{(\boldsymbol t+\ell)}{\sigma_{\! p}^{m}}$ implies $d\equiv_p\ell $ for all $d,c,\ell,m\in\N$. Without loss of generality we can assume that $m=\ell=0$. 
Observe that $\boldsymbol t=\shiftTuple{(\boldsymbol t+d)}{\sigma_{\! p}^c}$ implies
$\boldsymbol t_{(p,k\cdot c)}=\boldsymbol t_{(p,(k+1)\cdot c)}+d$ for all $k$.
Considering that $p\cdot c\equiv_{p}0$ we have 
\[\boldsymbol t_{(p,0)}=\boldsymbol t_{(p,c)}+d=
\dots=\boldsymbol t_{(p,(p-1)\cdot c)}+(p-1)\cdot d=\boldsymbol t_{(p,0)}+p\cdot d.\] 
Hence $p\cdot d\equiv_{q}0$. Since $p\notin Q$ we have that $p$ and $q$ are coprime. Therefore $d\equiv_{q}0$ as desired.

Repeating this for every $\boldsymbol t\in T$ defines $f$ on $\mathbb Q^{\Cyc p}$.
If $\boldsymbol r\stackrel{1}\to \boldsymbol s$, then $\boldsymbol s=(\boldsymbol r+1)=(+1)(\boldsymbol r)$ and $f(\boldsymbol r)\stackrel{1}\to f(\boldsymbol s)$, hence $f$ is a polymorphism of $\mathbb Q$. Furthermore $f=f_{\sigma_{\! p}}$ by definition.
\end{proof}

Understanding whether a loop condition implies another one is helpful for describing $\mPCL$ and $\PCPoset$.
This problem has already been studied before.
A sufficient condition is given in the following result from Ol\v{s}ak~\cite{olsak-loop}. 
\begin{theorem}[Corollary 1 in~\cite{olsak-loop}]\label{thm:loopolsak}
Let $\Sigma, \Gamma$ be loop conditions. If  $\mathbb{G}_{\Sigma}\to \mathbb{G}_{\Gamma}$, then $\Sigma\Rightarrow\Gamma$.
\end{theorem}

The idea of the proof is to show that if $h\colon \mathbb{G}_{\Sigma}\to \mathbb{G}_{\Gamma}$ and $f\models\Sigma$, then $f_h\models\Gamma$. 
We can use Theorem~\ref{thm:loopolsak} to easily deduce implications between cyclic loop conditions. However, for convenience, we state some results explicitly.
\begin{lemma}\label{lem:orderOnLoopcond} Let $C\subset\N^+\!$ be  finite and $c,d\in\N^+$\!.
\begin{enumerate}
    \item We have $\Sigma_C\Rightarrow\Sigma_{C\dotdiv c}$. In particular, $\Sigma_{C\dotdiv c}\Rightarrow \Sigma_{C\dotdiv (c\cdot d)}$.
    \item We have $\Sigma_C\Rightarrow\Sigma_{C\cupdot \{d\}}$.
    \item If $d$ is a multiple of an element of $C$, then $\Sigma_{C\cupdot \{d\}}\Leftrightarrow\Sigma_C$.
\end{enumerate}
\end{lemma}

Using these results we can characterize the orders of $\mPCL$ and $\PCPoset$.
\begin{lemma}\label{cor:implicationSinglePClc}
Let $\Sigma_P, \Sigma_Q$ be prime cyclic loop conditions. Then the following are equivalent:
\begin{multicols}{3}
\begin{enumerate}[label=(\arabic*)]
    \item $\Sigma_P\Rightarrow\Sigma_Q$,
    \item $\mathbb Q\leq \mathbb P$,
    \item $\Pol(\mathbb Q)\not\models\Sigma_P$,
    \item $P\subseteq Q$,
    \item $\mathbb P\to\mathbb Q$.
    \item[]
\end{enumerate}
\end{multicols}
\end{lemma}
\begin{proof}
We show (1) $\Rightarrow$ (3) $\Rightarrow$ (4)  $\Rightarrow$ (5)  $\Rightarrow$ (1) and (4)  $\Rightarrow$ (2) $\Rightarrow$ (3).

(1) $\Rightarrow$ (3)
Since $\Pol(\mathbb Q)\not\models\Sigma_Q$ we have, by assumption, that $\Pol(\mathbb Q)\not\models\Sigma_P$.

(3) $\Rightarrow$ (4)
We show the contraposition. 
Since $P\nsubseteq Q$ there is a $p\in P\setminus Q$. By Lemma~\ref{lem:duofpcSatisfyPclc} we have $\Pol(\mathbb Q)\models\Sigma_p$. By Lemma~\ref{lem:orderOnLoopcond} we have $\Sigma_p\Rightarrow\Sigma_P$. Hence $\Pol(\mathbb Q)\models\Sigma_P$.

(4) $\Rightarrow$ (5)
The identity-function is even an embedding from $\mathbb P$ to $\mathbb Q$.

(5) $\Rightarrow$ (1)
Follows from Theorem~\ref{thm:loopolsak}.

(4) $\Rightarrow$ (2)
Let $\C$ be the first pp-power of $\mathbb Q$ given by the formula $\Phi_E(x,y)\coloneqq x\stackrel{c}{\to}x\wedge x\to y$ where $c\coloneqq \prod_{p\in P}p$. 
The set of edges of $\C$ is the subset of edges of $\mathbb Q$ containing all edges that lie in a cycle whose length divides $c$.
Note that $q\in Q$ divides $c$ if and only if $q\in P$. Hence $\Cyc q\hookrightarrow\C$ if and only if $q\in P$.
Therefore, $\C$ consists of  $\mathbb P$ and possibly some isolated points and is homomorphically equivalent to $\mathbb P$.

(2) $\Rightarrow$ (3)
By Corollary \ref{cor:wond}, every height 1 identity satisfied by $\mathbb Q$ is also satisfied by $\mathbb P$. Since $\Pol(\mathbb P)\not\models\Sigma_P$ we have that $\Pol(\mathbb Q)\not\models\Sigma_P$.
\end{proof}

As a consequence of the previous corollary 
any element $[\Sigma_P]$ of $\mPCL$ can be represented by exactly one prime cyclic loop condition, i.e., $\Sigma_P$. Hence we will from now on identify $[\Sigma_P]\in\mPCL$ with $\Sigma_P$. Analogously, we identify $[\mathbb P]\in\PCPoset$ with $\mathbb P$.
Furthermore, we obtain a simple description of $\mPCL$ and $\PCPoset$.

\begin{corollary}\label{cor:mPclEqPPC}
$\mPCL\simeq\PCPoset\simeq (\{P\mid P\text{ a finite nonempty set of primes}\},\supseteq)$
\end{corollary}

\section{Cyclic loop conditions}\label{sec:conditions}
The goal of this section is to give a comprehensible description of the poset of cyclic loop conditions and of the poset of sets of cyclic loop conditions ordered by strength. The results from this section will help us to describe the poset of
pp-constructability types of finite disjoint unions of cycles, $\SDPoset$.
We start by giving a description of the implication order on cyclic loop conditions in Theorem~\ref{thm:implicationSingleClc}. Using a compactness argument we extend this description to  sets of cyclic loop conditions  in Corollary~\ref{thm:implicationClc}. 
We show that every cyclic loop condition is equivalent to a set of prime cyclic loop conditions. Finally, in Corollary~\ref{cor:characterizationMCLAndmCL}, we give a comprehensible description of the poset of sets of prime cyclic loop conditions ordered by their strength for clones, and hence for the poset of cyclic loop conditions ordered by their strength for clones.

\subsection{Single cyclic loop conditions}\label{ssc:clc}

\begin{definition}
We introduce the poset $\mCL$ as follows  
\begin{align*}
    \mCL&\coloneqq (\{[\Sigma] \mid \Sigma \text{ a cyclic loop condition}\},\Leftarrow).
\end{align*}
\end{definition}

One way to show that one cyclic loop condition is stronger than another is presented in the following example.

\begin{example}
Let $\mathcal{C}$ be a clone with domain $D$ such that $\mathcal C\models \Sigma_4$. Then there is some $f\in\mathcal C$ satisfying 
\[f(a,b,c,d)=f(b,c,d,a)\text{ for all }a,b,c,d\in D.\]
The function $g\colon (a,b)\mapsto f(a,b,a,b)$ satisfies
\[g(a,b)=g(b,a)\text{ for all }a,b\in D.\]
Hence $\mathcal C\models \Sigma_2$. Therefore $\Sigma_4\Rightarrow\Sigma_2$.
\eoe
\end{example}

Alternatively, $\Sigma_4\Rightarrow\Sigma_2$ follows from Theorem \ref{thm:loopolsak} and the fact that $\Cyc4\to\Cyc2$.
However, unlike for prime cyclic loop conditions, not every implication between cyclic loop conditions can be shown using Theorem~\ref{thm:loopolsak} as seen in the following example. 

\begin{example}\label{exa:2implies4}
Let $\mathcal{C}$ be a clone with domain $D$ such that $\mathcal C\models \Sigma_2$. Then there is some $g\in\mathcal C$ satisfying 
\[g(a,b)=g(b,a)\text{ for all }a,b\in D.\]
The function $f\colon (a,b,c,d)\mapsto g(g(a,b),g(c,d))$ satisfies the identity
\begin{align*}
    f(a,b,c,d)&=g(g(a,b),g(c,d))\\&=g(g(a,b),g(d,c))
    \\&= g(g(d,c),g(a,b))=f(d,c,a,b)
\end{align*}

 for all $a,b,c,d\in D$.
 Note that the digraph corresponding to this identity is isomorphic to $\Cyc4$.
 Hence $\mathcal C\models \Sigma_4$. Therefore $\Sigma_2\Rightarrow\Sigma_4$. However $\mathbb{G}_{\Sigma_2}=\Cyc 2 \not\to \Cyc 4 = \mathbb{G}_{\Sigma_4}$.
\eoe
\end{example}

The next goal is to weaken the condition in Theorem~\ref{thm:loopolsak} such that the converse implication also holds. First we generalize Example~\ref{exa:2implies4} by proving that $\{\Sigma_C,\Sigma_D\} \Rightarrow\Sigma_{C\cdot D}$ in Lemma~\ref{cor:bulletLoopConditions}. This lemma generalizes Proposition 2.2 in~\cite{Barto_modularity} 
from cyclic loop conditions with only one cycle to cyclic loop conditions in general. To prove this lemma we introduce some notation.  
The function $f$ constructed in Example~\ref{exa:2implies4} is the so-called \emph{star product} of $g$ with itself \cite{absorption}.

\begin{definition}
Let $A$ be a set, $f\colon A^n\to A$ and $g\colon A^m\to A$ be maps. The \emph{star product} of $f$ and $g$ is the function $(f\star g)\colon A^{n\cdot m}\to A$ defined by
\[(x_1,\dots,x_{n\cdot m})\mapsto f(g(x_1,\dots,x_m),\dots,g(x_{(n-1)\cdot m+1},\dots,x_{n\cdot m})).\]
For functions $f\colon A^I\to A$ and $g\colon A^J\to A$, where $I$ and $J$ are finite sets, we define the star product $(f\star g)\colon A^{I\mathbin\times J}\to A$ by
\[\boldsymbol t\mapsto f(i\mapsto g(j\mapsto  t_{(i,j)})).\]
\end{definition}
Note that the second definition extends the first one in the following sense. Let $f\colon A^I\to A$ and $g\colon A^J\to A$ be functions, where $I=[n]$ and $J=[m]$. 
Define $\tilde f\colon A^n\to A$ and $\tilde g\colon A^m\to A$ as $\tilde f(t_1,\dots,t_n)\coloneqq f(\boldsymbol t)$ and $\tilde g(s_1,\dots,s_m)\coloneqq g(\boldsymbol s)$ for all $\boldsymbol t\in A^I,\boldsymbol s\in A^J$. For $\sigma\colon I\mathbin\times J\to[n\cdot m]$ with $\sigma(i,j)=(i-1)\cdot m+j$ we have 
\[(f\star g)_\sigma(\boldsymbol t)=(\tilde f\star \tilde g)(t_1,\dots,t_{n\cdot m}).\]

With the star product we can easily show that $\{\Sigma_C,\Sigma_D\}\Rightarrow \Sigma_{\C\times \D}$. 
\begin{example}\label{exa:CDimpliesCtimesD}
Let $\mathcal{C}$ be a clone such that $\mathcal{C}\models \Sigma_C$ and $\mathcal{C}\models \Sigma_D$. Then there are $f,g\in\mathcal C$ with $f\models \Sigma_C$ and $g\models\Sigma_D$.
Then 
\[(f\mathbin\star g)(\boldsymbol t)= f((a,k)\mapsto g((b,\ell)\mapsto t_{((a,k),(b,\ell))})).\]

Observe that $f\mathbin\star g$ satisfies
\[f\mathbin\star g=f_{\sigma_{\!C}}\mathbin\star g=f_{\sigma_{\!C}}\mathbin\star g_{\sigma_{\!D}}=(f\mathbin\star g)_{\sigma_{\C\mathbin\times\D}}.\]

Hence $f\mathbin\star g\models \Sigma_{\C\mathbin\times\D}$.  Therefore $\{\Sigma_C,\Sigma_D\}\Rightarrow \Sigma_{\C\times \D}$. 
\eoe
\end{example}

The following example points out the difference between Example~\ref{exa:2implies4} and the special case of $C=D=\{2\}$ in Example~\ref{exa:CDimpliesCtimesD}.
\begin{example}\label{exa:TimesvsBullet}
Let $g\models\Sigma_2$. Then the digraph corresponding to the identity 
\[g\mathbin\star g=g(g,g)=g_{\sigma_2}(g_{\sigma_2},g_{\sigma_2})=g_{\sigma_2} \mathbin\star g_{\sigma_2}\] 
is $\Cyc2\mathbin\times\Cyc2$, which consists of two cycles of length 2. Hence $g\mathbin\star g$ satisfies $\Sigma_{\Cyc2\mathbin\times\Cyc2}$, which is equivalent to $\Sigma_2$.
Observe that in order to show $\Sigma_2\Rightarrow\Sigma_4$, in Example~\ref{exa:2implies4}, we used the identity 
\[g(g,g)=g_{\sigma_2}(g,g_{\sigma_2}).\] 
This identity corresponds to a digraph $\mathbb G$ isomorphic to $\Cyc4$. Hence $g\mathbin\star g$ also satisfies $\Sigma_{\mathbb G}$, which is equivalent to $\Sigma_4$. Note that $\mathbb G$ and $\Cyc2\mathbin\times\Cyc2$ have the same vertices.
\eoe
\end{example}

We now introduce a new edge relation on $\C\mathbin\times\D$, which in the case of $\C=\D=\Cyc2$ yields the graph $\mathbb G$ from Example~\ref{exa:TimesvsBullet}.

\begin{definition}
Let $C,D\subset\N^+\!$ be finite.
Define $\C\mathbin\bullet\D$ as the finite disjoint union of cycles $(V,E)$ with  $V=\C\times\D$ and  $E=\{(\boldsymbol t,\sigma_{\C\mathbin\bullet\D}(\boldsymbol t))\mid \boldsymbol t\in \C\times\D\}$, where
\[\sigma_{\C\mathbin\bullet\D}((a,k),(b,\ell))\coloneqq    
\begin{cases}
    ((a,0),(b,\ell+1))&\text{if }k = a-1\\
    ((a,k+1),(b,\ell)) & \text{otherwise.}
\end{cases}\]
Define 
$\C^{\bullet 0}\coloneqq\Cyc1$ and $\C^{\bullet (k+1)}\coloneqq\C \mathbin\bullet \C^{\bullet k}$ for $k\in\N$.
\end{definition}
For example, $\Cyc{2}\mathbin\bullet\Cyc{2}$ is isomorphic to $\Cyc{4}$. 

Observe that the set associated to  $\C\mathbin\bullet\D$ is 
\[C\cdot D\coloneqq \{a\cdot b\mid a\in C, b\in D\},\]
whereas, the set associated to the product $\C\times\D$ is $\{\lcm(a,b)\mid a\in C,b\in D\}$.

\begin{lemma}\label{lem:CbulletDequalsCD}
Let $C,D\subset \N^+\!$ finite. Then $\{\Sigma_C,\Sigma_D\}\Rightarrow \Sigma_{C\cdot D}$.
\end{lemma}
\begin{proof}
Let $\mathcal C$ be a clone such that $\mathcal C\models\Sigma_{C}$ and $\mathcal C\models\Sigma_{D}$. Then there are $f,g\in\mathcal C$ with $f\models \Sigma_{C}$ and $g\models \Sigma_{D}$.
Let $\{a_1,\dots,a_n\}=C$.
For any $i\in[n]$ define the following permutations on $\C\mathbin\bullet\D$
\begin{align*}
\tau\colon ((a,k),(b,\ell))&\mapsto 
((a,k+1),(b,\ell))
\\
\tau_{i}\colon ((a,k),(b,\ell))&\mapsto
\begin{cases}
((a,k),(b,\ell+1)) & \text{if }a= a_i, k=a-1 \\
((a,k),(b,\ell)) & \text{otherwise.}
\end{cases}
\end{align*}

\begin{figure}
    \centering
    \begin{tikzpicture}
        \node at (0+90:0.5) [scale=0.4] {$((3,0),(2,0))$};
        \node at (0,-1.1) [scale=0.4] {$((2,0),(2,0))$};
        \cycleEmpty{3}{(0,0)}{{1,2,3}}
        \cycleEmpty{2}{(0,-1.5)}{{1,2}}
        
        \node at ($(1.6,0)+(0+90:0.5)$) [scale=0.4] {$((3,0),(2,1))$};
        \node at (1.6,-1.1) [scale=0.4] {$((2,0),(2,1))$};
        \cycleEmpty{3}{(1.6,0)}{{1,2,3}}
        \cycleEmpty{2}{(1.6,-1.5)}{{1,2}}
    
    \end{tikzpicture}
    \hspace{1.9cm}
    \begin{tikzpicture}
        \cycleEmpty{3}{(0,0)}{{}}
        \cycleEmpty{2}{(0,-1.5)}{{}}
        
        \cycleEmpty{3}{(1.6,0)}{{}}
        \cycleEmpty{2}{(1.6,-1.5)}{{}}
    \def \radiusA {{3*0.035+0.15}}
    \def \radiusB {{2*0.035+0.15}}
        
        \path[->,>=stealth']
            (-30:\radiusA) edge[bend right,shorten >=0.5mm,shorten <=0.5mm] ($(1.6,0)+(-30:\radiusA)$)
            ($(1.6,0)+(-30:\radiusA)$) edge[bend right=40,shorten >=0.5mm,shorten <=0.5mm] (-30:\radiusA)
            ($(0,-1.5)+(90+180:\radiusB)$) edge[bend right,shorten >=0.5mm,shorten <=0.5mm] ($(1.6,-1.5)+(90+180:\radiusB)$)
            ($(1.6,-1.5)+(90+180:\radiusB)$) edge[bend right,shorten >=0.5mm,shorten <=0.5mm] ($(0,-1.5)+(90+180:\radiusB)$)
            ;
            \draw[>=stealth',arrows=-{>[bend]}] ($(1.6+0.05,0)+(90+120:\radiusA)$)
    arc ({-60}:{240}:1mm);
            \draw[>=stealth',arrows=-{>[bend]}] ($(1.6+0.05,0)+(90:\radiusA)$)
    arc ({-60}:{240}:1mm);
            \draw[>=stealth',arrows=-{>[bend]}] ($(0.05,0)+(90+120:\radiusA)$)
    arc ({-60}:{240}:1mm);
            \draw[>=stealth',arrows=-{>[bend]}] ($(0.05,0)+(90:\radiusA)$)
    arc ({-60}:{240}:1mm);
            \draw[>=stealth',arrows=-{>[bend]}] ($(0.05,-1.5)+(90:\radiusB)$)
    arc ({-60}:{240}:1mm);
            \draw[>=stealth',arrows=-{>[bend]}] ($(1.6+0.05,-1.5)+(90:\radiusB)$)
    arc ({-60}:{240}:1mm);
    \end{tikzpicture}
    \hspace{1.9cm}
    \begin{tikzpicture}

        \cycleEmpty{3}{(0,0)}{{1,2}}
        \cycleEmpty{2}{(0,-1.5)}{{1}}
        
        \cycleEmpty{3}{(1.6,0)}{{1,2}}
        \cycleEmpty{2}{(1.6,-1.5)}{{1}}
        
        \def \radiusA {{3*0.035+0.15}}
        \def \radiusB {{2*0.035+0.15}}

        \path[->,>=stealth']

            (90+240:\radiusA) edge[out=45,in=170,shorten >=0.5mm,shorten <=0.5mm] ($(1.6,0)+(90:\radiusA)$)

            ($(0,-1.5)+(90+180:\radiusB)$) edge[out=45,in=170,shorten >=0.5mm,shorten <=0.5mm] ($(1.6,-1.5)+(90:\radiusB)$)
        ;
        \path 
        (90:\radiusA) edge[<-,>=stealth',out=45,in=170,shorten <=0.5mm] ($(1.6,0)+(90:\radiusA)+(90:\radiusA)-(90+240:\radiusA)$)
        
        ($(1.6,0)+(90:\radiusA)+(90:\radiusA)-(90+240:\radiusA)$) edge[out=-10,in=50,shorten >=0.5mm] ($(1.6,0)+(90+240:\radiusA)$)

        ($(0,-1.5)+(90:\radiusB)$) edge[<-,>=stealth',out=45,in=170,shorten <=0.5mm] ($(1.6,-1.5)+(90:\radiusB)+(90:\radiusB)-(90+180:\radiusB)$)
        
        ($(1.6,-1.5)+(90:\radiusB)+(90:\radiusB)-(90+180:\radiusB)$) edge[out=-10,in=20,shorten >=0.5mm] ($(1.6,-1.5)+(90+180:\radiusB)$)
        
        ;
    
    \end{tikzpicture}
    \caption{Graphs of the permutations $\tau$ (left), $\tau_{1}\circ\tau_{2}$ (middle), and $\tau\circ\tau_{1}\circ\tau_{2}$ (right) on $\Cyc{2,3}\mathbin\bullet\Cyc{2}$.
}
    \label{fig:CBulletD}
\end{figure}
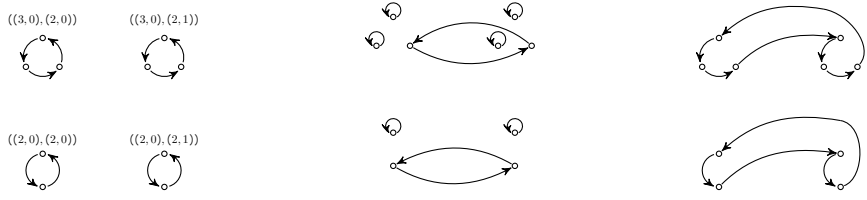
Observe that, since $f=f_{\sigma_{\!C}}$ and $g=g_{\sigma_{\!D}}$, we have $f\star g=(f\star g)_\tau=(f\star g)_{\tau_i}$ for all $i\in [n]$. 
We show that
\[\tau\circ\tau_{1}\circ\cdots\circ\tau_{n}
=\sigma_{\C\mathbin\bullet \D}.\]
First, the reader can verify that this equality holds for the example in Figure~\ref{fig:CBulletD}.
To prove that the equality holds in general let $(a,k)\in\C$ and $(b,\ell)\in\D$.
We have that
\begin{align*}
    (\tau\circ\tau_{1}\circ\cdots\circ\tau_{n})((a,k),(b,\ell))
    &=
    \begin{cases}
    \tau((a,k),(b,\ell+1))&\text{if }k=a-1\\
    \tau((a,k),(b,\ell)) & \text{otherwise}
    \end{cases}\\
    &=\begin{cases}
    ((a,0),(b,\ell+1))&\text{if }k=a-1\\
    ((a,k+1),(b,\ell)) & \text{otherwise}
    \end{cases}\\
    &=\sigma_{\C\mathbin{\bullet}\D}((a,k),(b,\ell)).
\end{align*}
Hence, $(f\star g)=(f\star g)_{\sigma_{\C\mathbin\bullet\D}}$ and since $C\cdot D$ is associated to $\C\mathbin\bullet\D$ there is, by Theorem~\ref{thm:loopolsak}, also an element in $\mathcal C$ that satisfies $\Sigma_{C\cdot D}$.
\end{proof}

Observe that $\C\mathbin\bullet\D\to\C$ and $\C\mathbin\bullet\D\to\D$. Hence, by Theorem~\ref{thm:loopolsak}, we have $\Sigma_{C\cdot D}\Rightarrow\Sigma_C$ and $\Sigma_{C\cdot D}\Rightarrow\Sigma_D$. Therefore, the implication in Lemma~\ref{lem:CbulletDequalsCD} is actually an equivalence. As a consequence we obtain the following corollary. 

\begin{corollary}\label{cor:bulletLoopConditions}
Let $C,C_1,\dots,C_n\subset \N^+\!$ be finite. Then 
$\{\Sigma_{C_1},\dots,\Sigma_{C_n}\}\Leftrightarrow\Sigma_{C_1\cdot\ldots\cdot C_n}$. In particular, $\Sigma_C\Leftrightarrow\Sigma_{C^k}$ for every $k\in\N^+\!$.
\end{corollary}

We now weaken the condition in Theorem~\ref{thm:loopolsak} such that it characterizes the implication order on $\mCL$. 
\begin{theorem}\label{thm:implicationSingleClc}

Let $C,D\subset\N^+\!$ finite. Then the following are equivalent:
\begin{enumerate}[label=(\arabic*)]
    \item We have $\Sigma_C\Rightarrow\Sigma_D$.
    \item For all $c\in \N^+\!$ we have $\Pol(\D\dotdiv c)\models\Sigma_C$ implies $\Pol(\D\dotdiv c)\models\Sigma_D$.
    \item For every $a\in C$ there exist $b\in D$ and $k\in\N$ such that $b$ divides $a^k$.
    \item There exists $k\in\N^+\!$ such that $\C^{\bullet k}\to \D$.
\end{enumerate}
\end{theorem}

In order to show (2) $\Rightarrow$ (3) we need to know more about the connection between disjoint unions of cycles and cyclic loop conditions. This connection is investigated later in Section~\ref{sec:structures}. As this section is devoted exclusively to cyclic loop conditions we will provide the proof of this implication later.

\begin{proof}
The direction (1) $\Rightarrow$ (2) is clear from the definition of the implication order. 

(2) $\Rightarrow$ (3) Follows from Lemma~\ref{lem:missingImplicationSingleClc}.

(3) $\Rightarrow$ (4) For $a\in C$ let $k_a\in\N$ be such that there is a $b\in D$ with $b$ divides  $a^{k_a}$. Let $n=|C|$ and $k=\max\{k_a\mid a\in C\}$. Note that any number in $C^{n\cdot k}$ must be a multiple of $a^k$ for some $a\in C$. Hence $\C^{\bullet (n\cdot k)}\to \D$.

(4) $\Rightarrow$ (1) 
Let $k\in\N^+\!$ be such that $\C^{\bullet k}\to \D$. Then
\[\Sigma_C\stackrel{\ast}\Leftrightarrow\Sigma_{C^k}\stackrel{\ast\ast}\Rightarrow\Sigma_{D},\]
where $\ast$ and $\ast\ast$ hold by Corollary~\ref{cor:bulletLoopConditions} and Theorem~\ref{thm:loopolsak}, respectively.
\end{proof}

\begin{remark} Note that it follows from (3) that the problem
\begin{align*}
    &\text{Input: two finite sets $C,D\subset\N^+$}\\
    &\text{Output: Does $\Sigma_C\Rightarrow\Sigma_D$ hold?}
\end{align*}
is decidable; in fact, this can be decided in polynomial time even if the integers in C and D are given in binary.
\end{remark}
Define the  function $\Flat$ from $\N^+$ to $\N^+$ as
\begin{align*}
    p_1^{\alpha_1}\cdot\ldots\cdot p_n^{\alpha_n}&\mapsto p_1\cdot\ldots\cdot p_n.
\end{align*}
From (3) in Theorem~\ref{thm:implicationSingleClc} we obtain that 
every cyclic loop condition is equivalent to a cyclic loop condition where only square-free numbers occur. 
\begin{corollary}\label{cor:flat}
For every finite $C\subset\N^+\!$ we have that $\Sigma_C\Leftrightarrow\Sigma_{\Flat(C)}$. 
In particular, we have $\Sigma_{\C\times\D}\Leftrightarrow\Sigma_{\C\mathbin\bullet\D}$.
\end{corollary}

\subsection{Sets of cyclic loop conditions}\label{ssc:setsOfClc}
The next goal is to understand the implication order on sets of cyclic loop conditions. 
\begin{definition}
We introduce the poset $\MCL$ as follows  

\begin{align*}
    \MCL&\coloneqq (\{[\Sigma]\mid \Sigma \text{ a set of cyclic loop conditions}\},\Leftarrow).
\end{align*}
\end{definition}

We already understand this order on finite sets of cyclic loop conditions: by Corollary \ref{cor:bulletLoopConditions}, every finite set of cyclic loop conditions is equivalent to a single cyclic loop condition and we know how to compare  single cyclic loop conditions by Theorem \ref{thm:implicationSingleClc}.
Using the compactness theorem for first-order logic we will show that in order to determine the order on infinite sets it suffices to consider their finite subsets. 
\begin{theorem}\label{thm:compactness}
Let $\Gamma$, $\Sigma$ be sets of height 1 identities, where $\Sigma$ is finite. Then $\Gamma\Rightarrow\Sigma$ if and only if there is a finite $\Gamma'\subseteq \Gamma$ such that $\Gamma'\Rightarrow\Sigma$.
\end{theorem}
\begin{proof}
Let $\Sigma=\{(f_1)_{\sigma_1}\approx (g_1)_{\tau_1},\dots,(f_k)_{\sigma_k}\approx (g_k)_{\tau_k}\}$ and let $\kappa$ be the set of function symbols occurring in $\Gamma$. We construct a first-order theory $T_{\Gamma,\Sigma}$ which is satisfiable if and only if $\Gamma\not\Rightarrow\Sigma$. 
The goal of the construction is that every model of $T_{\Gamma,\Sigma}$ encodes a clone which witnesses $\Gamma\not\Rightarrow\Sigma$ and 
conversely every clone which witnesses $\Gamma\not\Rightarrow\Sigma$ is a model of $T_{\Gamma,\Sigma}$.

The signature of $T_{\Gamma,\Sigma}$ is  $\kappa\cup\{D,F_n,E_n\mid n\in\N^+\}$, where
\begin{itemize}
    \item $\kappa$ is a set of constant symbols (intended to denote the operations satisfying $\Gamma$),
    \item $D$ is a unary relation symbol (intended to denote the domain),
    \item $F_n$ is a unary relation symbol (intended to denote the $n$-ary operations) for every $n\in\N^+\!$,
    \item $E_n$ is an $(n+2)$-ary relation symbol (intended to denote the graphs of all $n$-ary operations) for every $n\in\N^+\!$.
\end{itemize}

For the sake of readability we set the following abbreviations:
\begin{itemize}
    \item $\forall x\in R: \Phi(x)$ abbreviates $\forall x( R(x)\Rightarrow \Phi(x))$, for every unary relation symbol $R$
    \item $\exists x\in R: \Phi(x)$ abbreviates $\exists x( R(x)\wedge \Phi(x))$, for every unary relation symbol $R$
    \item $f(\boldsymbol x)\approx x_i$ abbreviates $E_n(f,x_1,\dots,x_n,x_i)$.
\end{itemize}
   
The theory $T_{\Gamma,\Sigma}$ consists of the following sentences
\begin{itemize}
    \item $\exists x (D(x))$,
    \item $\forall f\in F_n\ \forall x_1,\dots,x_n\in D\ \exists! y\in D :E_n(f,x_1,\dots,x_n,y)$ for every
    \item $\exists f\in F_n\ \forall \boldsymbol x\in D^n:f(\boldsymbol x)\approx x_i$ for every $n\in\N^+$ and $i\in[n]$,
    \item $\forall f\in F_n\ \forall g_1,\dots,g_n\in F_m\ \exists h\in F_m$
    \[\forall \boldsymbol x\in D^m :f(g_1(\boldsymbol x),\dots,g_n(\boldsymbol x))\approx h(\boldsymbol x)\] for every $n,m\in\N^+\!$,
    \item $F_n(f)$ for every $n$-ary $f\in \kappa$, $n\in\N^+\!$,
    \item $\forall x_1,\dots, x_r \in D: f(x_{\sigma(1)},\dots,x_{\sigma(n)})\approx g(x_{\tau(1)},\dots,x_{\tau(m)})$ for every identity $f_\sigma\approx g_\tau \in \Gamma$,
    \item $\forall f_1\in F_{n_1}\ \forall g_1\in F_{m_1}\dots \forall f_k\in F_{n_k}\ \forall g_k\in F_{m_k}:$
    \[\bigvee_{i=1}^k \exists x_1,\dots, x_{r_i} \in D:f_i(x_{\sigma_i(1)},\dots,x_{\sigma_i(n_i)})\not\approx g_i(x_{\tau_i(1)},\dots,x_{\tau_i(m_i)})\] 
    assume w.l.o.g. that $f_1,g_1,\dots,f_k,g_k\notin\kappa$.
\end{itemize}
Now we show that $T_{\Gamma,\Sigma}$ is unsatisfiable if and only if $\Gamma\Rightarrow\Sigma$.
If $\Gamma\not\Rightarrow\Sigma$, then there is a clone $\mathcal C$ over some domain $C$ such that $\mathcal C\models\Gamma$, witnessed by a function $\tilde\cdot\colon \tau\to\mathcal C$, and $\mathcal C\not\models \Sigma$. Let $\A$ be the structure with domain $C\cup \mathcal{C}$ and $D^{\A}\coloneqq C$, $F_n^\A\coloneqq\{f\in\mathcal C\mid f\text{ is $n$-ary}\}$, $E_n^\A\coloneqq\{(f,x_1,\dots,x_n,y)\mid f\in F_n^\A, x_1,\dots,x_n,y\in D, f(x_1,\dots,x_n)=y\}$, $f^\A\coloneqq\tilde f$ for every $n\in\N^+\!$ and every $f\in\kappa$. By construction $\A\models T_{\Gamma,\Sigma}$. Hence $T_{\Gamma,\Sigma}$ is satisfiable.

If $T_{\Gamma,\Sigma}$ is satisfiable, then it has some model $\A$. Let $C=D^\A$. For every $n\in\N^+\!$ and $f\in F_n^\A$ let $\bar f\colon C^n\to C$ be the operation with graph $\{(c_1,\dots,c_n,d)\mid (f,c_1,\dots,c_n,d)\in E_n^\A\}$. Define $\mathcal{C}\coloneqq\{\bar f\mid n\in\N^+, f\in F_n^\A\}$. By construction $\mathcal{C}$ is a clone over domain $C$ that satisfies $\Gamma$, witnessed by the assignment $f\mapsto \bar f^\A$, and does not satisfy $\Sigma$. Hence $\Gamma\not\Rightarrow\Sigma$.

If $\Gamma\Rightarrow\Sigma$ we have that theory $T_{\Gamma,\Sigma}$ is unsatisfiable. Hence, by compactness, there is a finite $T'\subseteq T$ that is also unsatisfiable. Since $T'$ is finite there must be a finite $\Gamma'\subseteq\Gamma$  such that $T'\subseteq T_{\Gamma',\Sigma}$. Hence, $T_{\Gamma',\Sigma}$ is unsatisfiable and $\Gamma'\Rightarrow\Sigma$.
\end{proof}

Note that the restriction to height 1 identities is not essential. The proof of Theorem~\ref{thm:compactness} can easily be adapted to the case where $\Gamma$ and $\Sigma$ are sets of identities (not necessarily height 1). 
From Theorem~\ref{thm:compactness},  Corollary~\ref{cor:bulletLoopConditions}, and Theorem~\ref{thm:implicationSingleClc} we obtain the following corollary.

\begin{corollary}\label{thm:implicationClc}
Let $\Gamma$ be a set of cyclic loop conditions and $D\subset\N^+\!$ be finite. Then the following are equivalent:
\begin{enumerate}[label=(\arabic*)]
    \item $\Gamma\Rightarrow\Sigma_D$.
    
    \item There is a finite $\Gamma'\subseteq\Gamma$ such that $\Gamma'\Rightarrow \Sigma_D$.
    \item There is a finite $\Gamma'\subseteq\Gamma$ such that for any $c\in \N^+\!$ we have $\Pol(\D\dotdiv c)\models\Gamma'$ implies $\Pol(\D\dotdiv c)\models\Sigma_D$.
    \item There are $\Sigma_{C_1},\dots,\Sigma_{C_n}\in\Gamma$ and $k_1,\dots,k_n\in\N^+\!$ with $\C_1^{\bullet k_1}\mathbin\bullet\cdots\mathbin\bullet\C_n^{\bullet k_n} \to \D$.
\end{enumerate}
\end{corollary}

Theorem~\ref{thm:implicationSingleClc} and~\ref{thm:implicationClc} provide a simple characterization of the implication order on $\mCL$ and $\MCL$, respectively. However, we did not yet achieve the goal of obtaining a comprehensible description of these posets.

\subsection{Irreducible cyclic loop conditions}\label{ssc:irredClc}

We take a second look at Corollary~\ref{cor:bulletLoopConditions}. It states in particular that the cyclic loop condition $\Sigma_{C_1\cdot\ldots\cdot C_n}$ is equivalent to the set $\{\Sigma_{C_1},\dots,\Sigma_{C_n}\}$. This suggests that one can replace a single cyclic loop condition by a set of possibly simpler cyclic loop conditions.

\begin{definition}
Let $\Sigma$ be a cyclic loop condition. A \emph{decomposition} of $\Sigma$ is a set of pairwise incomparable cyclic loop conditions $\Gamma$ such that $\Sigma$ and $\Gamma$ are equivalent. The condition $\Sigma$ is called \emph{irreducible} if every decomposition of $\Sigma$ contains exactly one element.
\end{definition}

\begin{example}\label{exa:decomp}
Some examples of decompositions.
\begin{multicols}{2}
\begin{itemize}
    \item $\Sigma_6 \Leftrightarrow\{\Sigma_2, \Sigma_3\}$
    \item $\Sigma_{30} \Leftrightarrow \{\Sigma_2, \Sigma_{15}\} \Leftrightarrow {}$\hbox to 1cm {$\{\Sigma_2, \Sigma_3, \Sigma_5\}$}
\item $\Sigma_{6,20} \Leftrightarrow \{ \Sigma_2, \Sigma_{3,5}\}$
\item     $\Sigma_{2,15} \Leftrightarrow \{\Sigma_{2,3},  \Sigma_{2,5}\}$
\end{itemize}
\end{multicols}

Note that all decompositions except for $\{\Sigma_2, \Sigma_{15}\}$ contain only irreducible cyclic loop conditions.
\eoe
\end{example}

Observe that all irreducible  cyclic loop conditions from the example are prime cyclic loop conditions.

\begin{lemma}\label{lem:pclcAreImpliedBySingleCond}
Let $\Sigma_P$ be a prime cyclic loop condition and let $\Gamma$ be a set of cyclic loop conditions such that $\Gamma\Rightarrow\Sigma_P$. Then there is a $\Sigma\in\Gamma$ with $\Sigma\Rightarrow\Sigma_P$.
\end{lemma}
\begin{proof}
By Corollary~\ref{cor:flat}, we may  assume all numbers occurring in conditions in $\Gamma$ are square-free. 

Since $\Gamma\Rightarrow\Sigma_P$ we have, by Corollary~\ref{thm:implicationClc}, that there are there are $\Sigma_{C_1},\dots,\Sigma_{C_n}\in\Gamma$  with $\C_1\mathbin\bullet\cdots\mathbin\bullet\C_n \to \mathbb P$.
Assume that $\C_i\not\to\mathbb P$ for any $i$. Then for every $i$ there is an $a_i\in\C_i$ such that no $p\in P$ divides $a_i$. Hence there is a cycle of length $a_1\cdot\ldots\cdot a_n$ in $\C_1\mathbin\bullet\cdots\mathbin\bullet\C_n$ and no $p\in P$ divides $a_1\cdot\ldots\cdot a_n$, a contradiction.
Therefore, there must be an $i$ such that $\C_i\to\mathbb P$. Hence, by Theorem~\ref{thm:loopolsak}, we have that $\Sigma_{C_i}\Rightarrow\Sigma_P$.
\end{proof}

\begin{corollary}\label{lem:pclcAreIrred}
Every prime cyclic loop condition is irreducible.
\end{corollary}

The next step is to show that every cyclic loop condition can be decomposed into a set of prime cyclic loop conditions. The next example gives the idea how to find this decomposition.

\begin{example}\label{exa:620}
Let us consider $\Sigma_{6,20}$ from Example \ref{exa:decomp} and let $\Gamma$ be a decomposition of $\Sigma_{6,20}$. Then every condition in $\Gamma$ is implied by $\Sigma_{6,20}$. 
By Lemma~\ref{lem:orderOnLoopcond}, we know that $\Sigma_{6,20}$ implies $\Sigma_{\{6,20\}\dotdiv c}$ for every $c\in\N^+\!$. All nontrivial conditions of this form are presented in Figure~\ref{fig:exampleDotdivPrimeConditions}. 
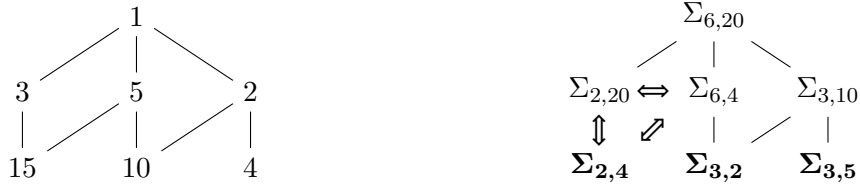
\begin{figure}
\begin{minipage}{0.48\textwidth}
     \centering
    \begin{tikzpicture}
    \node (01) at (0,-2) {$10$};
    
    \node (10) at (-1.5,-1) {$3$};
    \node (11) at (0,-1) {$5$};
    \node (12) at (1.5,-1) {$2$};
    
    \node (20) at (0,0) {$1$};
    \node (02) at (1.5,-2) {$4$};
    \node (00) at (-1.5,-2) {$15$};
    \path 
        (20) edge (10)
        (20) edge (11)
        (20) edge (12)
        
        (10) edge (00) 
        (11) edge (00) 
        (11) edge (01)
        (12) edge (01)
        (12) edge (02)
        ;
    \end{tikzpicture}
   \end{minipage}\hfill
   \begin{minipage}{0.48\textwidth}
     \centering
    \begin{tikzpicture}
    \node (01) at (0,-2) {$\boldsymbol{\Sigma_{3,2}}$};
    
    \node (10) at (-1.5,-1) {$\Sigma_{2,20}$};
    \node (11) at (0,-1) {$\Sigma_{6,4}$};
    \node (12) at (1.5,-1) {$\Sigma_{3,10}$};
    
    \node (20) at (0,0) {$\Sigma_{6,20}$};
    \node (02) at (1.5,-2) {$\boldsymbol{\Sigma_{3,5}}$};
    \node (00) at (-1.5,-2) {$\boldsymbol{\Sigma_{2,4}}$};
    
    \node at (-0.8,-1) {$\Leftrightarrow$};
    \node[rotate=90] at (-1.5,-1.5) {$\Leftrightarrow$};
    \node[rotate=45] at (-0.8,-1.5) {$\Leftrightarrow$};
    \path 
        (20) edge (10)
        (20) edge (11)
        (20) edge (12)
        (11) edge (01)
        (12) edge (01)
        (12) edge (02)
        ;
    \end{tikzpicture}
   \end{minipage}\hfill
   \caption{Some numbers ordered by divisibility (left) and corresponding cyclic loop conditions of the form $\Sigma_{\{6,20\}\dotdiv c}$ ordered by strength (right). }\label{fig:exampleDotdivPrimeConditions}
\end{figure}
We would like to emphasize the following two observations: 
\begin{enumerate}
    \item the conditions  $\Sigma_{2,4},\Sigma_{3,2}, \Sigma_{3,5}$ 
    printed in bold at the bottom of the figure
    are equivalent to prime cyclic loop conditions and
    \item $\Sigma_{6,20}\Leftrightarrow \{ \Sigma_2,\Sigma_{3,2}, \Sigma_{3,5}\}\Leftrightarrow \{ \Sigma_2, \Sigma_{3,5}\}$. \hfill$\triangle$
\end{enumerate} 
\end{example}

\begin{definition}\label{def:minimal}

Let $C\subseteq\N^+\!$ be a finite set. A number $c\in\N^+\!$ is \emph{maximal for} $C$ if $1\notin(C\dotdiv c)$ and $1\in{C\dotdiv (c\cdot d)}$ for all $d>1$ dividing $\lcm(C\dotdiv c)$.
\end{definition}

We now show that the two observations in Example~\ref{exa:620} hold in general.

\begin{lemma}\label{lem:minimalIsPrime}
Let $C\subseteq\N^+\!$ be a finite set and let  $c\in\N^+\!$ be maximal for $C$. 
Then $\Sigma_{C\dotdiv c}$ is equivalent to a prime cyclic loop condition.

\end{lemma}
\begin{proof}

We show that $\Sigma_{C\dotdiv c}$ is equivalent to a prime cyclic loop condition by showing that every $a\in (C\dotdiv c)$ is a multiple of some prime in $C\dotdiv c$.
Let $a$ be in $ (C\dotdiv c)$ and $p$ be a prime divisor of $a$ (which exists since $a\neq1$). 
Since $c$ is maximal for $C$ we have $1\notin(C\dotdiv c)$ and $1\in (C\dotdiv (c\cdot p))$. Hence, $p\in (C\dotdiv c)$ and  $a$ is a multiple of a prime in $ (C\dotdiv c)$, as desired.
\end{proof}

\begin{theorem}\label{thm:clcAreSetsOfPclc}
Every cyclic loop condition is equivalent to a finite set of prime cyclic loop conditions.
\end{theorem}
\begin{proof}
Let $C\subset\N^+\!$ be finite and consider the cyclic loop condition $\Sigma_C$. Without loss of generality assume that $C$ is square-free. 
Let $c_1,\dots, c_n\in\N^+\!$ be such that 
\[\{\Sigma_{C\dotdiv c_1},\dots,\Sigma_{C\dotdiv c_n}\}=\{\Sigma_{C\dotdiv c}\mid c\text{ is maximal for }C\}.\] 
By Lemma~\ref{lem:orderOnLoopcond}, we have that  $\Sigma_C\Rightarrow\{\Sigma_{C\dotdiv c_1},\dots,\Sigma_{C\dotdiv c_n}\}$.
From Lemma~\ref{lem:minimalIsPrime} we know that for every $i$ the condition  $\Sigma_{C\dotdiv c_i}$ is equivalent to a prime cyclic loop condition. 
Assume  for contradiction that 
\[(\C\dotdiv c_1)\mathbin\bullet\dots\mathbin\bullet(\C\dotdiv c_n)\not\to\C.\]
Hence, there are $a_1,\dots,a_n$ with $a_i\in C\dotdiv c_i$ such that no $a\in C$ divides $c := a_1\cdot\ldots\cdot a_n$.
Therefore, $1\notin C\dotdiv c$ and the cyclic loop condition $\Sigma_{C\dotdiv c}$ is not trivial.
Let $d\in \N^+\!$ be such that $c\cdot d$ is maximal for $C$. 
Then there is an $i$ such that $\Sigma_{C\dotdiv c\cdot d}= \Sigma_{C\dotdiv c_i}$. 
Note that the number $a_i$ can not be in $C\dotdiv c\cdot d$,  because $C$ is square-free. Therefore, $C\dotdiv c\cdot d\neq C\dotdiv c_i$ for all $i$. A contradiction.
By Corollary~\ref{thm:implicationClc} it follows that 
$\{\Sigma_{C\dotdiv c_1},\dots,\Sigma_{C\dotdiv c_n}\}\Rightarrow\Sigma_C$, 
as desired.
\end{proof}

Note that, as seen in Example~\ref{exa:620}, the set $\Gamma\coloneqq\{\Sigma_{C\dotdiv c}\mid c\text{ is maximal for }C\}$ is not necessarily a decomposition of $\Sigma_C$. However, the set consisting of the strongest conditions in $\Gamma$ is a decomposition of $\Sigma_C$. This observation together with Theorem~\ref{thm:clcAreSetsOfPclc} and Corollary~\ref{lem:pclcAreIrred} yields the following corollary.

\begin{corollary}
Let $\Sigma$ be a cyclic loop condition. Then
\begin{enumerate}
    \item $\Sigma$ is irreducible iff $\Sigma$ is equivalent to a prime cyclic loop condition and
    \item there is a  finite decomposition of $\Sigma$ which consists of prime cyclic loop conditions.
\end{enumerate}
\end{corollary}

The next step towards finding a comprehensible description of $\mCL$ and $\MCL$ is to understand 
the poset of sets of prime cyclic loop conditions ordered by strength.

\begin{definition}
We introduce the poset $\MPCL$ as follows
\begin{align*}
\MPCL&\coloneqq (\{[\Sigma]\mid \Sigma \text{ a set of prime cyclic loop conditions}\},\Leftarrow).
\end{align*}
\end{definition} 

From Theorem~\ref{thm:clcAreSetsOfPclc} we obtain the following corollary.
\begin{corollary}
We have that $\MCL=\MPCL$.
\end{corollary}
 
Combining Lemma~\ref{lem:pclcAreImpliedBySingleCond} and Corollary~\ref{cor:implicationSinglePClc} we obtain the following result.
\begin{corollary}\label{cor:implicationPclc}
Let $\Gamma$ be a set of prime cyclic loop conditions and $Q$ a finite nonempty set of primes. Then the following are equivalent:
\begin{enumerate}[label=(\arabic*)]
    \item We have $\Gamma\Rightarrow\Sigma_Q$.
    
    \item There is a $\Sigma_P\in\Gamma$ such that $\Sigma_P\Rightarrow \Sigma_Q$.
    \item There is a $\Sigma_P\in\Gamma$ such that  $\Pol(\mathbb Q)\not\models\Sigma_P$.
    \item There is a $\Sigma_{P}\in\Gamma$  with $P\subseteq Q$.
    \item There is a $\Sigma_{P}\in\Gamma$  with $\mathbb P \to \mathbb Q$.
    
\end{enumerate}
\end{corollary}

Let $\Gamma$ be a set of cyclic loop conditions and $\Gamma'\coloneqq\{\Sigma_P\in\mPCL\mid \Gamma\Rightarrow\Sigma_P \}$. Then $\Gamma\Leftrightarrow\Gamma'$ and $\Gamma'$ is a downset of $\mPCL$. Corollary~\ref{cor:implicationPclc} implies that if $\Gamma$ is a downset  of $\mPCL$, then $\Gamma=\Gamma'$. Hence the elements of $\MPCL$ correspond to downsets of $\mPCL$. 

\begin{corollary}\label{cor:characterizationMCLAndmCL}
We have that 
$\MCL=\MPCL\simeq (\operatorname{Downsets Of}(\mPCL),\subseteq)$ and $\mCL\simeq(\operatorname{Finitely Generated Downsets Of}(\mPCL),\subseteq)$.
\end{corollary}

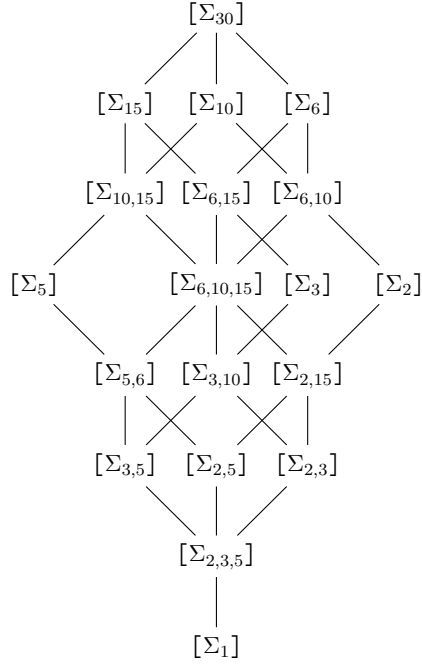
\begin{figure}
    \centering
        \begin{tikzpicture}[scale=0.6]
\def\myScale{0.8}

\node[scale=\myScale] (0) at (2,-2)  {$[\Sigma_{1}]$};
\node[scale=\myScale] (00) at (2,0)  {$[\Sigma_{2,3,5}]$};
\node[scale=\myScale] (11) at (0,2)  {$[\Sigma_{3,5}]$};
\node[scale=\myScale] (12) at (2,2) {$[\Sigma_{2,5}]$};
\node[scale=\myScale] (13) at (4,2) {$[\Sigma_{2,3}]$};
\node[scale=\myScale] (21) at (0,4) {$[\Sigma_{5,6}]$};
\node[scale=\myScale] (22) at (2,4) {$[\Sigma_{3,10}]$};
\node[scale=\myScale] (23) at (4,4) {$[\Sigma_{2,15}]$};
\node[scale=\myScale] (31) at (-2,6) {$[\Sigma_{5}]$};
\node[scale=\myScale] (33) at (4,6) {$[\Sigma_{3}]$};
\node[scale=\myScale] (34) at (6,6) {$[\Sigma_{2}]$};
\node[scale=\myScale] (32) at (2,6) {$[\Sigma_{6,10,15}]$};
\node[scale=\myScale] (41) at (0,8) {$[\Sigma_{10,15}]$};
\node[scale=\myScale] (42) at (2,8) {$[\Sigma_{6,15}]$};
\node[scale=\myScale] (43) at (4,8) {$[\Sigma_{6,10}]$};
\node[scale=\myScale] (51) at (0,10) {$[\Sigma_{15}]$};
\node[scale=\myScale] (52) at (2,10) {$[\Sigma_{10}]$};
\node[scale=\myScale] (53) at (4,10) {$[\Sigma_{6}]$};
\node[scale=\myScale] (60) at (2,12) {$[\Sigma_{30}]$};
\path 
    (0)  edge (00)
    (00) edge (11)
    (00) edge (12)
    (00) edge (13)
    (11) edge (21)
    (11) edge (22)
    (12) edge (21)
    (12) edge (23)
    (13) edge (22)
    (13) edge (23)
    (21) edge (31)
    (21) edge (32)
    (22) edge (32)
    (22) edge (33)
    (23) edge (32)
    (23) edge (34)
    (31) edge (41)
    (32) edge (41)
    (32) edge (42)
    (32) edge (43)
    (33) edge (42)
    (34) edge (43)
    (41) edge (51)
    (41) edge (52)
    (42) edge (51)
    (42) edge (53)
    (43) edge (52)
    (43) edge (53)
    (51) edge (60)
    (52) edge (60)
    (53) edge (60)
    ;
\end{tikzpicture}
    \caption{The poset $\mCL$ restricted to cyclic loop conditions using 2, 3, and 5.}
    \label{fig:conditionPoset}
\end{figure}
A subposet of $\mCL$ is represented in Figure~\ref{fig:conditionPoset}.
We will discuss more properties of the posets $\mCL$ and $\MPCL$ in Section~\ref{sec:lattice}.

\section{Unions of cycles and cyclic loop conditions}\label{sec:structures}

In this section we describe $\SDPoset$, namely the subposet of $\PPPoset$ containing the pp-constructability types of disjoint unions of cycles.

Firstly, in Lemma~\ref{lem:duofcSatisfyClc} we characterize $\models$ on disjoint unions of cycles and cyclic loop conditions. Using this characterization we prove the missing implication in Theorem~\ref{thm:implicationSingleClc}.
Secondly, in Lemma~\ref{thm:PPvsLoop}, we show that cyclic loop conditions suffice to
determine the pp-constructability order on finite disjoint unions of cycles.
This allows us to view a finite disjoint union of cycles as the set of prime cyclic loop conditions that it satisfies. These sets are always downsets of $\mPCL$. In Lemma~\ref{lem:surjectivety} we show when a downset is realized by a finite disjoint union of cycles. Combining these results we obtain a complete description of $\SDPoset$ as downsets of $\mPCL$.

\subsection{Characterization of the satisfaction relation for unions of cycles and cyclic loop conditions}\label{ssc:models}
In Lemma~\ref{lem:duofcSatisfyClc} we generalize Lemma~\ref{lem:duofpcSatisfyPclc} to disjoint unions of cycles and cyclic loop conditions. 
First, recall that in Example~\ref{exa:C3notmodelsS3} we showed that $\Cyc3\not\models\Sigma_3$ by constructing a suitable tuple. We adopt the same idea in the following example.
\begin{example}\label{exa:C56notmodelsS25}
We show that the structure $\Cyc{5, 6}$ does not satisfy the cyclic loop condition $\Sigma_{2,5}$.
Define
\[\boldsymbol t\coloneqq ((6,0),(6,3),(5,0),(5,3),(5,1),(5,4),(5,2)).\] 
Note that $\boldsymbol t$ is contained in a cycle of length 30 in $\Cyc{5,6}^{\Cyc{2,5}}$. Since $\boldsymbol t+3=\boldsymbol t_{\sigma_{2,5}}$ any $f\in\Pol(\Cyc{5,6})$ with $f\models\Sigma_{2,5}$ satisfies
\begin{align*}
f(\boldsymbol t)
=f_{\sigma_{2,5}}(\boldsymbol t)
=f(\shiftTuple{\boldsymbol t}{\sigma_{2,5}})
=f(\boldsymbol t+3).
\end{align*}
Therefore, $f$ would map $\boldsymbol t$ to an element of $\Cyc{5,6}$ that lies in a cycle whose length divides 3. Hence, such an $f$ does not exist and $\Cyc{5,6}\not\models\Sigma_{2,5}$. 
\eoe
\end{example}

The following lemma shows that the line of reasoning used in Example~\ref{exa:C56notmodelsS25} is the only obstacle for a cyclic loop condition to be satisfied by a disjoint union of cycles.
It characterises the relation $\models$ on disjoint unions of cycles and cyclic loop conditions. Therefore, it will be a main tool in many of the subsequent proofs in this article.

\begin{lemma}\label{lem:duofcSatisfyClc}
Let $\C$ be a finite disjoint union of cycles and $D\subset\N^+\!$ finite. We have $\Pol(\C)\models \Sigma_D$ if and only if for all maps $h\colon D\to C$ there is an $a\in C$ such that
\[a\text{ divides }\lcm(\{h(b)\dotdiv b \mid {b\in D}\}).\]
\end{lemma}
The proof of this lemma is very similar to that of Lemma \ref{lem:duofpcSatisfyPclc}, it just has more technical difficulties.

\begin{proof}
Assume without loss of generality that $\C$ is the finite disjoint union of cycles associated to the set $C\subset \N^+$\!.

$(\Rightarrow)$
Consider a map $h\colon D\to C$ and a polymorphism $f\colon \C^\D\to\C$ of $\C$ satisfying $\Sigma_D$. Define 
\[c \coloneqq \lcm(\{h(b)\dotdiv b\mid{b\in D}\})\]
and the tuple $\boldsymbol t\in\C^\D$ as $\boldsymbol t_{(b,k)}\coloneqq (h(b),c\cdot k)$ for $(b,k)\in\D$.
We will show $f(\boldsymbol t)=f(\boldsymbol t+c)$.
Observe that $(\shiftTuple{\boldsymbol t}{\sigma_{\! D}})_{(b,k)}=\boldsymbol t_{\sigma_{\! D}(b,k)}=\boldsymbol t_{(b,k+1)}$ for all $(b,k)\in\D$.
Let $b\in D$. We have, by the definition of $c$, that $h(b)$ divides $c\cdot b$. Hence, $c\cdot b\equiv_{h(b)}0$ and 
\[(\boldsymbol t_{\sigma_{\! D}})_{(b,b-1)}=\boldsymbol t_{(b,0)}=(h(b),0)=(h(b),c\cdot b)=(h(b),c\cdot (b-1)+c)= \boldsymbol t_{(b,b-1)}+c.\] 
Therefore $\shiftTuple{\boldsymbol t}{\sigma_{\! D}}=\boldsymbol t+c$, which just says that  $\boldsymbol t$ maps neighbouring points in $\D$ to points in $\C$ that are connected with a path of length $c$.
Furthermore, we have
\begin{align*}
f(\boldsymbol t)
=f_{\sigma_{\! D}}(\boldsymbol t)
=f(\shiftTuple{\boldsymbol t}{\sigma_{\! D}})
=f(\boldsymbol t+c).
\end{align*}

Since $f$ is a polymorphism we have that $f(\boldsymbol t) \stackrel{c}{\to}f(\boldsymbol t+c)=f(\boldsymbol t)$. Hence $f(\boldsymbol t)$ is in a cycle whose length divides $c$.

$(\Leftarrow)$
For this direction we construct a polymorphism $f\colon \C^\D \to \C$ of $\C$ satisfying $\Sigma_D$.
Let $\textbf G$ be the subgroup of the symmetric group on $\C^\D$ generated by $\sigma_{\! D}\colon \boldsymbol t \mapsto \shiftTuple{\boldsymbol t }{\sigma_{\! D}}$ and $+1\colon \boldsymbol t\mapsto (\boldsymbol t+1)$. Recall that
\begin{align*}
    (\shiftTuple{\boldsymbol t} {\sigma_{\! D}})_{(b,k)}=\boldsymbol t_{(b,k+1)},&& +1=\sigma_{\C^\D}, && ((+1)(\boldsymbol t))_{(b,k)}=\boldsymbol t_{(b,k)}+1.
\end{align*}
Hence, $\sigma_{\! D}\circ (+1)=(+1)\circ \sigma_{\! D}$, $\textbf G$ is commutative, and every element of $\textbf G$ is of the form $\sigma_{\! D}^c\circ (+1)^d$ for some $c, d \in\N$.
For every orbit of $\C^\D$ under $\textbf G$ pick a representative and denote the set of representatives by $T$. Let $\boldsymbol t\in T$ and define the map $h\colon D\to C$ such that for every $b\in D$ there is a $k\in \mathbb Z_{h(b)}$ with $\boldsymbol t_{(b,0)}=(h(b),k)$. 
By assumption there is an $a_{\boldsymbol t}\in C$ such that
\begin{equation}
a_{\boldsymbol t}\text{ divides } \lcm(\{h(b)\dotdiv b\mid {b\in D}\}).\label{equ:aDivides}
\end{equation}
Note that the orbit of $\boldsymbol t$ is a disjoint union of cycles of a fixed length $n$ and that $a_{\boldsymbol t}$ divides $n$. Define $f$ on the orbit of $\boldsymbol t$ as 
\[f\left((\sigma_{\! D}^c\circ (+1)^d)(\boldsymbol t)\right)
=f\left(\shiftTuple {(\boldsymbol t+d)}{\sigma_{\! D}^c}\right)
\coloneqq (a_{\boldsymbol t},d)\text{ for all $c,d$.}\] 
To show that $f$ is well defined on the orbit of $\boldsymbol t$ it suffices to prove that $\shiftTuple{(\boldsymbol t+d)}{\sigma_{\! D}^c}=\shiftTuple{(\boldsymbol t+\ell)}{\sigma_{\! D}^{m}}$ implies $d\equiv_{a_{\boldsymbol t}}\ell $ for all $d,c,\ell,m\in\N$. Without loss of generality we can assume that $m=\ell=0$. 
Fix some $b\in D$. We want to show that $(h(b)\dotdiv b)$ divides $d$.
Observe that $\boldsymbol t=\shiftTuple{(\boldsymbol t+d)}{\sigma_{\! D}^c}$ implies
$\boldsymbol t_{(b,k\cdot c)}=\boldsymbol t_{(b,(k+1)\cdot c)}+d$ for all $k$.
Considering that $(b\dotdiv c)\cdot c\equiv_{b}0$ we have 
\[\boldsymbol t_{(b,0)}=\boldsymbol t_{(b,c)}+d=\boldsymbol t_{(b,2\cdot c)}+2\cdot d=\dots=\boldsymbol t_{(b,0)}+(b\dotdiv c)\cdot d.\] 
Hence $(b\dotdiv c)\cdot d\equiv_{h(b)}0$ and there is some $k\in\N$ such that $d=\frac{k\cdot h(b)}{b\dotdiv c}$. Therefore $h(b)\dotdiv (b\dotdiv c)$ divides $d$ and also $h(b)\dotdiv b$ divides $d$. Since this holds for all $b\in D$ we have, by \eqref{equ:aDivides}, that $a_{\boldsymbol t}$ divides $d$ as desired.

Repeating this for every $\boldsymbol t\in T$ defines $f$ on $\C^\D$. The function $f$ is well defined since the orbits partition  $\C^\D$.
If $\boldsymbol r\stackrel{1}\to \boldsymbol s$, then $\boldsymbol s=(\boldsymbol r+1)=(+1)(\boldsymbol r)$ and $f(\boldsymbol r)\stackrel{1}\to f(\boldsymbol s)$, hence $f$ is a polymorphism of $\C$. Furthermore $f=f_{\sigma_{\! D}}$ by definition.
\end{proof}

\begin{example}
Some applications of Lemma~\ref{lem:duofcSatisfyClc}:
\begin{itemize}
    \item $\Cyc{10} \models\Sigma_{2,5}$, since $h(2)=h(5)=10$ is the only map from $\{2,5\}$ to $\{10\}$ and 10 divides $10=\lcm(h(2)\dotdiv2,h(5)\dotdiv 5)$,
    \item $\Cyc{n}\not\models\Sigma_n$ for $n>1$, witnessed by $h(n)=n$, since $n$ does not divide $1=h(n)\dotdiv n$,
    \item $\Cyc{5, 6} \not\models\Sigma_{2,5}$, witnessed by $h(2)=6$, $h(5)=5$, since neither 5 nor 6 divide $3=\lcm(h(2)\dotdiv2,h(5)\dotdiv 5)$.
\hfill$\triangle$
\end{itemize}
\end{example}

As a consequence of Lemma~\ref{lem:duofcSatisfyClc} we obtain the following lemma.
\begin{lemma}\label{lem:AdoesntSatisfyOwnLoops}
Let $\C$ be a finite disjoint union of cycles and $c\in\N^+$\!. Then
\begin{align*}
    \Pol(\C)\models\Sigma_{C\dotdiv c} &&\text{if and only if}&&1\in(C\dotdiv c).
\end{align*}
\end{lemma}
\begin{proof}
The $\Leftarrow$ direction is clear. If $1\in(C\dotdiv c)$, then $\Sigma_{C\dotdiv c}$ is trivial since it is satisfied by the projection $\pi\colon \C^{\C\dotdiv c}\to \C$, $\boldsymbol{t}\mapsto \boldsymbol{t}_{(1,0)}$.

For the $\Rightarrow$ direction  
let $h\colon (C\dotdiv c)\to C$ be some map with $h(b)\dotdiv c=b$ for every $b\in C\dotdiv c$. For instance, $h(b)=\min\{a\in C\mid a\dotdiv c=b \}$.
Note that $h(d\dotdiv c)=d$ for all $d\in\operatorname{Im}(h)$. We apply Lemma~\ref{lem:duofcSatisfyClc} to $h$ and obtain some $\Cyc{a}\hookrightarrow \C$ such that $a$ divides 
\begin{align*}
\lcm(\{h(b)\dotdiv b \mid {b\in (C\dotdiv c)}\})&=
\lcm(\{d\dotdiv(d\dotdiv c)\mid{d\in \operatorname{Im}(h)}\})\\
&=\lcm(\{\gcd(d,c)\mid {d \in  \operatorname{Im}(h)}\})    
\end{align*}
which divides $c$. Therefore, $(a\dotdiv c)=1\in (C\dotdiv c)$.
\end{proof}

Note that we did not use any of the results from Section~\ref{sec:conditions} to prove Lemmata~\ref{lem:duofcSatisfyClc} and~\ref{lem:AdoesntSatisfyOwnLoops}.
Using these lemmata we can prove the missing implication from Theorem~\ref{thm:implicationSingleClc}. 

\begin{lemma}\label{lem:missingImplicationSingleClc}
Let $C,D\subset\N^+\!$ finite. We have that (2) implies (3).
\begin{enumerate}[label=(\arabic*)]
\setcounter{enumi}{1}
\item For all $c\in \N^+\!$ we have $\Pol(\D\dotdiv c)\models\Sigma_C$ implies $\Pol(\D\dotdiv c)\models\Sigma_D$.
\item For every $a\in C$ there exist $b\in D$ and $k\in\N$ such that $b$ divides $a^k$.
\end{enumerate}
\end{lemma}

\begin{proof}
We show the contraposition. Let $a\in C$ be such that no $b\in D$ divides $a^k$ for any $k\in\N$. Define $c\coloneqq a^k$ where $k$ is the highest prime power of any number in $D$, i.e., $k$ is the largest $\ell\in\N$ for which there is a prime $p$ and an $b\in D$ such that $p^{\ell}$ divides $b$. Consider the structure $\D\dotdiv c$. By construction $1\notin (D\dotdiv c)$. Hence, by Lemma~\ref{lem:AdoesntSatisfyOwnLoops}, $\Pol(\D\dotdiv c)\not\models\Sigma_D$.

Using Lemma~\ref{lem:duofcSatisfyClc} we prove  $\Pol(\D\dotdiv c)\models\Sigma_C$. Let $h\colon C\to D\dotdiv c$ be a map. Choose any $b\in D$ such that $h(a)=b\dotdiv c$. Then, by construction of $c$, 
\[h(a)=b\dotdiv c=b\dotdiv a^k=b\dotdiv a^k\dotdiv a=h(a)\dotdiv a\]
and $h(a)$ divides $\lcm(\{h(\tilde a)\dotdiv \tilde a\mid{\tilde a\in C}\})$. Hence, by Lemma~\ref{lem:duofcSatisfyClc}, $\Pol(\D\dotdiv c)\models\Sigma_C$ and $\Pol(\D\dotdiv c)\not\models\Sigma_D$.
\end{proof}

\subsection{On the pp-constructability of  unions of cycles}\label{ssc:ppOrder}

Now we have all the necessary ingredients to prove the connection between cyclic loop conditions and pp-constructions in $\SDPoset$ stated in Lemma~\ref{thm:PPvsLoop}.
In particular,  we
show that cyclic loop conditions suffice to separate disjoint unions of cycles.
We suggest to look at the following concrete pp-constructions first.
Recall that for every $k\in\N^+$ we abbreviate the pp-formula
\[\exists y_1,\dots,y_{k-1}( E(x,y_1)\wedge E(y_1,y_2)\wedge\dots\wedge E(y_{k-1},z))\]
by $x\stackrel{k}\to y$.
\begin{example}\label{ex.components}
The digraph $\Cyc{2,3}$ pp-constructs $\Cyc6$. Consider the second pp-power of $\Cyc{2,3}$ given by the pp-formula:
\[\Phi_E(x_1,x_2,y_1,y_2)\coloneqq (x_1\stackrel{1}{\to}y_1) \wedge (x_2\stackrel{1}{\to}y_2) \wedge (x_1\stackrel{2}{\to}x_1) \wedge (x_2\stackrel{3}{\to}x_2).\]
The resulting structure, which consists of one copy of $\Cyc6$ and 19 isolated points, is homomorphically equivalent to $\Cyc6$, and therefore $\Cyc{2,3}\leq\Cyc6$.
\eoe
\end{example}

\begin{example}\label{ex.jakubconstruction}
The digraph $\Cyc3$ pp-constructs $\Cyc9$. Consider the third pp-power of $\Cyc3$ given by the formula:
\begin{align*}
    \Phi_E(x_1,x_2,x_3,y_1,y_2,y_3)\coloneqq (x_2\approx y_1) \wedge 
    (x_3\approx y_2)\wedge (x_1\stackrel{1}{\to}y_3).
    \tag{$\ast$}
\end{align*}
Let us denote the resulting structure by $\C$. There is an edge $\boldsymbol s\stackrel{1}{\to}\boldsymbol t$ in $\C$ if the tuple $\boldsymbol t$ is obtained from $\boldsymbol s$ by first increasing the first entry and then shifting all entries cyclically, see Figure~\ref{fig:incShift}. 
\begin{figure}
    \centering
    \begin{tikzpicture}
    \node at (-0.7,0) {$\boldsymbol s=$};
    \node at (0,0) {$\begin{pmatrix}
    a\\b\\c
    \end{pmatrix}$};
    \node at (3,0) {$\begin{pmatrix}
    a+1\\b\\c
    \end{pmatrix}$};
    \node at (6,0) {$\begin{pmatrix}
    b\\c\\a+1
    \end{pmatrix}$};
    \node at (7,0) {$=\boldsymbol t$};
    
    \path[->] 
    (0.5,0) edge node[above] {increase} (2.3,0)
    (3.7,0) edge node[above] {shift} (5.3,0);
    \end{tikzpicture}
    \caption{The shape of tuples $\boldsymbol s$ and $\boldsymbol t$ in the edge-relation defined by the pp-formula $(\ast)$.
    }
    \label{fig:incShift}
\end{figure}
With this it is clear that for every element $\boldsymbol t$ in $\C$ we have $\boldsymbol t\stackrel{9}{\to}\boldsymbol t$. 
It turns out that $\C$ consists of three copies of $\Cyc9$, hence $\Cyc3\leq\Cyc9$ and even $\Cyc3\equiv\Cyc9$.

Note that the third pp-power of $\Cyc2$ given by the formula $(\ast)$ is not homomorphically equivalent to $\Cyc6$; instead, it is isomorphic to $\Cyc{2,6}$, which is homomorphically equivalent to $\Cyc2$.
\eoe
\end{example}

Although it is neither clear nor necessary we would like to mention that the pp-construction in the proof of the following lemma is essentially just a combination of the three constructions we saw in the Examples~\ref{ex.division},~\ref{ex.components} and~\ref{ex.jakubconstruction}. 

\begin{lemma}\label{thm:PPvsLoop}
Let $\C$ be a finite disjoint union of cycles and let $\B$ be a finite structure with finite relational signature $\tau$. Then 
\begin{align*}
\B\leq \C && \text{iff} && \Pol(\B)\models\Sigma_{C\dotdiv c} \text{ implies } \Pol(\C)\models\Sigma_{C\dotdiv c}\text{ for all $c$ that divide $\lcm(C)$.}
\end{align*}


\end{lemma}
We remark that the first part of the following proof is a specific instance of the proof of 
$ \F_{\Pol(\B)}(\A)\to\A$ implies  $\B\leq\A$
in Theorem~\ref{thm:freestructure}.
As the reader might not be familiar with free structures, we present a self-contained proof.
\begin{proof}
We show both directions separately.

$(\Rightarrow)$ 
Since $\Sigma_{C\dotdiv c}$ is a height $1$ condition, this direction follows 
from Corollary~\ref{cor:wond}.
$(\Leftarrow)$ Assume without loss of generality that $\C$ is the structure associated to $C$.  
Let $\F$ be the $\left|\B^{\C}\right|$-th pp-power of $\B$ defined by the formula
\begin{align*}
&\Phi_E(x,y)\coloneqq\bigwedge\left\{
x_{\scalebox{0.77}{$\shiftTuple{\boldsymbol t}{{\sigma_{\! C}}}$}}\approx 
y_{\boldsymbol t} \wedge \Phi_R(x)\ \middle|\ \boldsymbol t\in \B^\C, R\in\tau \right\}\!,
\intertext{where for every $k$-ary   $R\in\tau$ we have}
&\Phi_R(x)\coloneqq\bigwedge\left\{R(x_{\boldsymbol t_1},\dots,x_{\boldsymbol t_k})\ \middle|\ \boldsymbol t_1,\dots,\boldsymbol t_k\in\B^\C,\text{ $( t_{1u},\dots, t_{ku})\in R^\B$ for all $u\in\C$}\right\}\!.
\end{align*}
We can think of the elements of $\F$ as maps from $\B^\C$ to $\B$. 
The formula $\Phi_R(f)$ holds if and only if $f$ preserves $R^\B$. Note that $f$ preserves $R^\B$ if and only if $f_{\sigma_{\!C}}$ preserves $R^\B$. Hence $\Phi_E$ ensures that all elements of $\F$ that are not polymorphisms of $\B$ are isolated points. On the other hand polymorphisms $f$ of $\B$ that are in $\F$ have exactly one in-neighbour, namely $f_{\sigma^{-1}_{\! C}}$, and one out-neighbour, namely $f_{\sigma_{\! C}}$. 
Hence, $\F$ is homomorphically equivalent to a disjoint union of cycles, i.e., the structure $\F$ without isolated points. 
Furthermore, all cycles in $\F$ are of the form 
\[ f\stackrel{1}\to f_{\sigma_{\! C}}\stackrel{1}\to f_{\sigma^2_{\! C}}\stackrel{1}\to \dots \stackrel{1}\to f_{\sigma^k_{\! C}}=f \]
for some $k\in\N$.

We show that $\F$ and $\C$ are homomorphically equivalent by proving the following two statements:
\begin{enumerate}
\item For all $a\in\N^+\!$ we have $\Cyc{a}\hookrightarrow \C$ implies $\Cyc{a} \hookrightarrow \F$ and
\item for all $c\in\N^+\!$ we have $\Cyc{c} \hookrightarrow \F$ implies $\Cyc{c}\to \C$.
\end{enumerate}

First statement: Suppose that $\Cyc{a}\hookrightarrow \C$. Then the polymorphism $\pi_{(a,0)}\colon\B^\C\to\B$, $ \boldsymbol t\mapsto \boldsymbol t_{(a,0)}$ lies in the following cycle of length $a$ in $\F$:
\[\pi_{(a,0)}\stackrel{1}\to\pi_{(a,1)}\stackrel{1}\to\dots\stackrel{1}\to\pi_{(a,a-1)}\stackrel{1}\to\pi_{(a,0)}.\]

Second statement: Suppose that $\Cyc{c}\hookrightarrow \F$ and let $f$ be a polymorphism in a cycle of length $c$ in $\F$. Then $f=f_{\sigma^c_{\! C}}$. Let $\D$ be the $c$-th relational power of $\C$. Note that $\sigma_\D=\sigma^c_{\! C}$. Hence,  $f\models \Sigma_\D$. By Lemma~\ref{lem:relPowEqualsCdotdivc}, the digraphs $\D$ and $\C\dotdiv c$ are homomorphically equivalent. Therefore, $\Pol(\B)\models\Sigma_{C\dotdiv c}$ and, by assumption, $\Pol(\C)\models\Sigma_{C\dotdiv c}$ as well. Applying Lemma~\ref{lem:AdoesntSatisfyOwnLoops} we conclude that $1\in(C\dotdiv c)$. Hence, there is some $a\in C$ such that $a$ divides $c$ and $\Cyc{c}\to \Cyc{a}\hookrightarrow \C$.

It follows that $\F$ and $\C$ are homomorphically equivalent. Hence, $\C$ is pp-constructable from $\B$.
\end{proof}

The construction in the proof was discovered by Jakub Opr\v{s}al (see~\cite{jakubPCSP} or~\cite{oprsal18} for more details). We thank him for explaining it to us.
Note that, following the notation introduced in Definition~\ref{def:freestructure}, the structure $\F$, after removing all isolated points, is $\F_{\Pol(\B)}(\C)$.

\begin{example}\label{ex.6,20,15}
Suppose we want to test whether a structure $\B$ pp-constructs $\Cyc{6,20,15}$.
By Lemma~\ref{lem:AdoesntSatisfyOwnLoops}, $\Pol(\Cyc{6,20,15})$ does not satisfy any non-trivial loop conditions of the form $\Sigma_{\{6,20,15\}\dotdiv c}$. Hence, by Lemma~\ref{thm:PPvsLoop}, to verify that $\B\leq \Cyc{6,20,15}$ we only have to check whether $\B$ satisfies none of the cyclic loop conditions $\Sigma_{6,20,15}$, $\Sigma_{2,20,5}$, $\Sigma_{3,10,15}$, $\Sigma_{6,4,3}$, $\Sigma_{3,5,15}$, $\Sigma_{3,2,3}$, which are the non-trivial ones of the form $\Sigma_{\{6,20,15\}\dotdiv c}$. By Theorem \ref{thm:implicationSingleClc}, these conditions are equivalent to  $\Sigma_{6,10,15}$, $\Sigma_{2,5}$, $\Sigma_{3,10}$, $\Sigma_{2,3}$, $\Sigma_{3,5}$, and $\Sigma_{2,3}$, respectively.

We show that the disjoint union of cycles $\Cyc{2,3,5}$ can pp-construct $\Cyc{6,20,15}$. First we check that $\Cyc{2,3,5}\not\models \Sigma_{6,10,15}$. Consider the map $h(6)=2$, $h(10)=2$, $h(15)=3$. We have
\[\lcm(2\dotdiv 6, 2\dotdiv 10, 3\dotdiv 15)=\lcm(1,1,1)=1.\]
Clearly, neither $2$ nor $3$ nor $5$ divide 1. Hence, by Lemma~\ref{lem:duofcSatisfyClc}, $\Cyc{2,3,5}\not\models \Sigma_{6,10,15}$. Similarly, $\Cyc{2,3,5}$ does not satisfy the other four loop conditions. Therefore, $\Cyc{2,3,5}\leq\Cyc{6,20,15}$. 
On the other hand, the structure $\Cyc{2,15}$ cannot pp-construct $\Cyc{6,20,15}$ since it satisfies  $\Sigma_{3,5}$. 
\eoe
\end{example}

By  Theorem~\ref{thm:clcAreSetsOfPclc}  every cyclic loop condition is equivalent to a set of prime cyclic loop conditions; we thus obtain that even prime cyclic loop conditions suffice to separate disjoint unions of cycles.

\begin{corollary}\label{thm:classificationPPCon}
Let $\mathbb C$ be a finite disjoint union of cycles and $\B$ be a finite structure with finite relational signature. Then
\begin{align*}
    \B\leq\mathbb C && \text{iff} && \Pol(\B)\models\Sigma\text{ implies }\Pol(\C)\models\Sigma\text{ for all prime cyclic loop conditions $\Sigma$.}
\end{align*}


\end{corollary}
Note that we can also prove this corollary using easier results. The cyclic loop conditions that are considered in Lemma~\ref{thm:classificationPPCon} to test whether $\B$ pp-constructs $\C$ have the form $\Sigma_{C\dotdiv c}$. To verify the condition in the lemma it suffices to check whether $\B$ does not satisfy any condition that is minimal in $\{\Sigma_{C\dotdiv c}\mid c\in\N^+\!, 1\notin(C\dotdiv c)\}$. 
Any such minimal condition $\Sigma_{C\dotdiv c}$ has a $c$ that is maximal for $C$ and is, by Lemma~\ref{lem:minimalIsPrime}, equivalent to a prime cyclic loop condition.

\subsection{Characterizing unions of cycles in terms of prime cyclic loop conditions}\label{ssc:SDPoset}
As a consequence of Corollary~\ref{thm:classificationPPCon}, every element $[\C]$ of $\SDPoset$ is uniquely determined by the set of prime cyclic loop conditions that $\C$ satisfies. Hence the map  
\[\PL \colon \C \mapsto \{\Sigma\mid \text{$\Sigma$ a prime cyclic loop condition and } \Pol(\C)\models\Sigma \}\] 
is injective. 
The next goal is to determine the image of $\PL$. 
First we simplify the characterization from Lemma~\ref{lem:duofcSatisfyClc} for prime cyclic loop conditions.

\begin{lemma}\label{lem:duofcSatisfyPclc}
Let $\C$ be a finite disjoint union of cycles and let $\Sigma_P$ be a prime cyclic loop condition. Then the following are equivalent:
\begin{enumerate}[label=(\arabic*)]
    \item $\Pol(\C)\not\models\Sigma_P$.
    \item There is a $c\in\N^+\!$ such that $1\notin (C\dotdiv c)$ and $P\subseteq (C\dotdiv c)$.
    \item There is a $c\in\N^+\!$ such that $\Sigma_{C\dotdiv c}$ is non-trivial and $\Sigma_P\Rightarrow\Sigma_{C\dotdiv c}$.
\end{enumerate} 
\end{lemma}
\begin{proof}
(1) $\Rightarrow$ (2) Since $\Pol(\C)\not\models\Sigma_P$, by Lemma~\ref{lem:duofcSatisfyClc}, there is a map $h\colon P\to C$ such that no $a\in C$ divides $c$, where $c\coloneqq \lcm(\{h(p)\dotdiv p\mid {p\in P}\})$.
In particular, every $p$ divides $h(p)$.
Note that if $a\dotdiv c=1$, then $\gcd(a,c)=a$ and  $a$ divides $c$.
Hence, we have $1\notin(C\dotdiv c)$. 
Let $p\in P$. Then $h(p)$ does not divide $c$, hence $h(p)\dotdiv c\neq1$. Furthermore, $h(p)\dotdiv c$ divides $h(p)\dotdiv (h(p)\dotdiv p)=\gcd(h(p),p)=p$.
Therefore,  $h(p)\dotdiv c=p$ for all $p$ and $P\subseteq (C\dotdiv c)$.

The directions
(2) $\Rightarrow$ (3) and (3) $\Rightarrow$ (1) follow from Lemma~\ref{lem:orderOnLoopcond} and  Lemma~\ref{lem:AdoesntSatisfyOwnLoops}, respectively.
\end{proof}

\begin{lemma}\label{lem:surjectivetyUpperBound}
For every disjoint union of cycles $\C$ we have that $\PL(\C)$ is a cofinite downset of $\mPCL$.
\end{lemma}
\begin{proof}
Let $P$ be the set of all prime divisors of $\lcm(C)$. Then, by Lemma~\ref{lem:duofcSatisfyPclc}, every prime cyclic loop condition $\Sigma_S$ with $\Pol(\C)\not\models\Sigma_S$ satisfies $S\subseteq P$. Since $P$ is finite there are only finitely many prime cyclic loop conditions that are not satisfied by $\C$. 
\end{proof}

Next we show that every cofinite downset of $\mPCL$ is also realized by some disjoint union of cycles.
\begin{lemma}\label{lem:surjectivety}
Let $\Gamma$ be a cofinite downset of $\mPCL$ and $\Downset_{\min}$ the set of minimal prime cyclic loop conditions of $\mPCL\setminus\Downset$. Then 
\[\C\coloneqq\bigtimes_{\Sigma\in\Downset_{\min}} \mathbb G_\Sigma\]
is a finite disjoint union of cycles and $\PL(\C)=\Gamma$.
\end{lemma}
\begin{proof}
Let $P$ denote the set $\bigcup\{T\mid \Sigma_T\in\Downset_{\min}\}$, which contains the primes occurring in $\Downset_{\min}$. 
Since $\Downset$ is cofinite we have that $\Downset_{\min}$ is finite. Hence, $\C$ is a finite disjoint union of cycles and
\[ C= \left\{
\lcm(\{p_T\mid { \Sigma_T\in\Downset_{\min}}\})
~\middle|~ 
\text{$p_T\in T$ for every $\Sigma_T\in\Downset_{\min}$}\right\}\!.\] 
We prove that $\PL(\C)=\Downset$. 

$(\subseteq)$  
Let $\Sigma_S\in\mPCL\setminus\Downset$. Since $\PL(\C)$ is closed under implication, we can assume $\Sigma_S$ to be minimal in $\mPCL\setminus\Downset$. Define $c_S\coloneqq \prod(P\setminus S)$. 
We want to apply Lemma~\ref{lem:duofcSatisfyPclc} to show $\Pol(\C)\not\models\Sigma_S$. 
Firstly, note that, since $\Sigma_S\in\Downset_{\min}$, any $a\in C$ is a multiple of some prime $p\in S$. Furthermore, this $p$ does not divide $c_S$, hence $a\dotdiv c_S\not=1$ and $1\notin (C\dotdiv c_S)$. 
Secondly, let $p\in S$. Since $\Sigma_S$ is minimal we have that for every other $\Sigma_T\in\Downset_{\min}$ there exists a $p_T\in T\setminus S$. Define $p_S\coloneqq p$ and $a\coloneqq \lcm(\{p_T\mid { \Sigma_T\in\Downset_{\min}}\})$. Then $a\in C$ and $a\dotdiv c_S=p$. Therefore $p\in (C\dotdiv c_S)$. Hence, $S\subseteq (C\dotdiv c_S)$ and, by Lemma~\ref{lem:duofcSatisfyPclc}, $\Pol(\C)\not\models \Sigma_S$ as desired.

$(\supseteq)$ Let $\Sigma_S\in\mPCL\setminus\PL(\C)$. Since $\Pol(\C)\not\models\Sigma_S$, by Lemma~\ref{lem:duofcSatisfyPclc}, we have that $S$ is contained in a set of the form $(C\dotdiv c)$.
Next we show that there is some $\Sigma_{S_c}\in\Downset_{\min}$ such that $S_c$ is contained in $(C\dotdiv c)$. 
Assume for contradiction that for every $\Sigma_T\in\Downset_{\min}$ the set $T$ is not contained in $(C\dotdiv c)$. Let $p_T$ be a witness of this fact.  
Note that $p_T\notin (C\dotdiv c)$ implies $p_T$ divides $c$. Then $a\coloneqq \lcm(\{p_T\mid { \Sigma_T\in\Downset_{\min}}\})\in C$ but $a\dotdiv c=1$, a contradiction. Hence, there is a $\Sigma_{S_c}\in\Downset_{\min}$ such that $S_c\subseteq (C\dotdiv c)$.

We show that $S\subseteq S_c$. Let $p\in S\subseteq (C\dotdiv c)$. Then there is some $a\in C$ such that $a\dotdiv c=p$. Again $a$ is of the form $\lcm(\{p_T\mid{ T\in\Downset_{\min}}\})$. Note that, since all numbers in $C$ are square-free, no element from $S_c$ can divide $c$. Hence, $p=p_{S_c}\in S_c$.
Therefore $S\subseteq S_c$ and $\Sigma_S$ implies $\Sigma_{S_c}$. Since $\Downset$ is implication-closed and $\Sigma_{S_c}\notin\Gamma$ we conclude that $\Sigma_S\notin \Downset$. This yields $\PL(\C)=\Downset$, as desired.
\end{proof}

The following two corollaries are immediate consequences of Corollary~\ref{thm:classificationPPCon}, Lemma~\ref{lem:surjectivetyUpperBound}, and Lemma~\ref{lem:surjectivety}. 

\begin{corollary}
Let $\C$ be a finite disjoint union of cycles. Then there is  a  finite  disjoint  union  of  cycles $\D$ whose  cycle  lengths  are  square-free  such  that $[\C]=[\D]$.
\end{corollary}

We finally obtain a classification of $\SDPoset$.
\begin{corollary}
\label{thm:classificationPoset}
The map
\begin{align*}
    \SDPoset&\to\MPCL\\
    [\C]&\mapsto[\PL(\C)]
\end{align*}
is well defined and an embedding of posets. Its image consists of the elements that can be represented by a cofinite downset of $\mPCL$. Put differently, \[SDPoset\simeq(\operatorname{CofiniteDownsetsOf}(\mPCL),\subseteq).\]
\end{corollary}

To better understand what the map from Corollary~\ref{thm:classificationPoset} does, have a look at the illustration in Figure~\ref{fig:example23}.
\begin{figure}
    \centering
    \begin{tikzpicture}[scale=0.6]

\node (0) at (2,-2)  {$[\Cyc{2,3}]$};
\node (00) at (2,0)  {$[\Cyc6]$};
\node (11) at (0,2)  {$[\Cyc2]$};
\node (13) at (4,2) {$[\Cyc3]$};
\node (22) at (2,4) {$[\Cyc1]$};

\node at (6.2,1) {$\mapsto$};
\path 
    (0)  edge (00)
    (00) edge (11)
    (00) edge (13)
    (11) edge (22)
    (13) edge (22)
    ;
\end{tikzpicture}
\hspace{5mm}
\begin{tikzpicture}[scale=0.6]

\node (0) at (2,-2)  {$[\mPCL\setminus\left\{\Sigma_{2},\Sigma_{3},\Sigma_{2,3}\right\}]$};
\node (00) at (2,0)  {$[\mPCL\setminus\left\{\Sigma_{2},\Sigma_{3}\right\}]$};
\node (11) at (-0.2,2)  {$[\mPCL\setminus\left\{\Sigma_{2}\right\}]$};
\node (13) at (4.2,2) {$[\mPCL\setminus\left\{\Sigma_{3}\right\}]$};
\node (22) at (2,4) {$[\mPCL]$};

\path 
    (0)  edge (00)
    (00) edge (11)
    (00) edge (13)
    (11) edge (22)
    (13) edge (22)
    ;
\end{tikzpicture}
\caption{ The embedding from Corollary~\ref{thm:classificationPoset} restricted to disjoint unions of cycles of length 2 and 3.}
    \label{fig:example23}
\end{figure}
We can give an explicit description of $\PL(\C)$.

\begin{lemma}\label{lem:pclCharacterization}
Let $\C$ be a finite disjoint union of cycles. Then
\[\PL(\C)=\mPCL\setminus\operatorname{UpsetOf}(\{\Sigma_{C\dotdiv c}\mid c\text{ is maximal for $C$}\}).\]
\end{lemma}
\begin{proof}
($\subseteq$) 
By Lemma~\ref{lem:AdoesntSatisfyOwnLoops}, we have that $\C$ does not satisfy $\Sigma_{C\dotdiv c}$ for any $c$ that is maximal for $C$. Hence no condition in  $\operatorname{UpsetOf}(\{\Sigma_{C\dotdiv c}\mid c\text{ is maximal for $C$}\})$ is satisfied by $\C$.

($\supseteq$) Let $\Sigma_P$ be a prime cyclic loop condition such that $\Pol(\C)\not\models\Sigma_P$. Then, by Lemma~\ref{lem:duofcSatisfyPclc}, there is a $c\in\N^+\!$ such that $\Sigma_{C\dotdiv c}$ is non-trivial and $\Sigma_P\Rightarrow\Sigma_{C\dotdiv c}$. Note that we can choose $c$ to be maximal for $C$. Hence the condition $\Sigma_P$ is in $\operatorname{UpsetOf}(\{\Sigma_{C\dotdiv c}\mid c\text{ is maximal for $C$}\})$.
\end{proof}

For a given finite disjoint union of cycles, Lemmata~\ref{lem:surjectivety} and~\ref{lem:pclCharacterization}
provide a method to construct a finite disjoint union of cycles of square-free length with the same pp-constructability type which can be carried out by hand on small examples, as illustrated by the following example.
\begin{example}\label{exa:620Structure}
Consider the structure $\Cyc{6,20}$.  The set from Lemma~\ref{lem:pclCharacterization} containing all conditions of the form $\Sigma_{\{6,20\}\dotdiv c}$ with $c$ maximal for $\{6,20\}$ is $\{\Sigma_{2,4},\Sigma_{3,2},\Sigma_{3,5}\}$, which is equivalent to   $\{\Sigma_{2},\Sigma_{3,2},\Sigma_{3,5}\}$.
Note that $\Sigma_2\Rightarrow\Sigma_{3,2}$. Hence,
\[\PL(\Cyc{6,20})\stackrel{\ref{lem:pclCharacterization}}{=}\operatorname{UpsetOf}(\{\Sigma_{2},\Sigma_{3,2},\Sigma_{3,5}\})=\operatorname{UpsetOf}(\{\Sigma_{3,2},\Sigma_{3,5}\})\stackrel{\ref{lem:surjectivety}}{=}\PL(\Cyc{3,10}).\]
Therefore, $\Cyc{6,20}\equiv\Cyc{3,10}$. 
\eoe
\end{example}

We have a similar situation for cyclic loop conditions: Corollary \ref{cor:flat} in particular states that every cyclic loop condition is equivalent to one that only uses square-free numbers.
Recall from Example~\ref{exa:620} that \[\Sigma_{6,20}\Leftrightarrow\{\Sigma_{2},\Sigma_{3,2},\Sigma_{3,5}\}\Leftrightarrow\{\Sigma_{2},\Sigma_{3,5}\}\Leftrightarrow\Sigma_{6,10}.\]

The different behaviors of disjoint unions of cycles and cyclic loop conditions when constructing square-free representatives is explained by the following observations. 
Let $C\subset\N$ finite and 
\[S=\{\{ p\in (C\dotdiv c)\mid p\text{ is prime}\}\mid \text{$c$ is maximal for $C$}\}.\]  
Let $S_{\min}$ ($S_{\max}$) be the set of minimal (maximal) elements of $S$ with respect to inclusion. Then  
\begin{itemize}
    \item $\C$ has the same pp-constructability type as $\bigtimes_{P\in S_{\min}}\mathbb P$,
    \item $\Sigma_C\Leftrightarrow\{\Sigma_P\mid P\in S_{\max}\}\Leftrightarrow\Sigma_{\Flat(C)}$, and
    \item if $C$ contains only square-free numbers, then $S=S_{\min}=S_{\max}$ and $S$ is an antichain.
\end{itemize}
Compare this to  Example~\ref{exa:620Structure}, where $C=\{6,20\}$. In this case we have that $S=\{\{2\},\{3,2\},\{3,5\}\}$, $S_{\min}=\{\{3,2\},\{3,5\}\}$, and $S_{\max}=\{\{2\},\{3,5\}\}$.

Recall that every finite smooth digraph either pp-constructs $\KThree$ or is homomorphically equivalent to a finite disjoint union of cycles (Theorem~\ref{thm:BartoKozikNiven}).
The following corollary is an immediate consequence of this fact together with Corollary~\ref{thm:classificationPoset} and Lemma~\ref{lem:surjectivety}.

\begin{corollary}\label{cor:cyclesSquarefree}
Let $\mathbb{G}$ be a finite smooth digraph. Then either $[\mathbb G]=[\KThree]$ or there is a finite disjoint union of cycles $\C$ whose cycle lengths are square-free such that $[\mathbb G]= [\C]$.
\end{corollary}

\section{The lattices of disjoint unions of cycles and  cyclic loop conditions}\label{sec:lattice}

The characterizations from Corollary~\ref{cor:characterizationMCLAndmCL} and  Theorem~\ref{thm:classificationPoset} suggest that the posets $\MPCL$ and $\SDPoset$ can be described lattice-theoretically. 
\begin{figure}
    \centering
    \begin{tikzpicture}[scale=0.6]
\def\myScale{0.8}

\node[scale=\myScale] (0) at (2,-2)  {$[\Cyc{2,3,5}]$};
\node[scale=\myScale] (00) at (2,0)  {$[\Cyc{6,10,15}]$};
\node[scale=\myScale] (11) at (0,2)  {$[\Cyc{2,15}]$};
\node[scale=\myScale] (12) at (2,2) {$[\Cyc{3,10}]$};
\node[scale=\myScale] (13) at (4,2) {$[\Cyc{5,6}]$};
\node[scale=\myScale] (21) at (0,4) {$[\Cyc{10,15}]$};
\node[scale=\myScale] (22) at (2,4) {$[\Cyc{6,15}]$};
\node[scale=\myScale] (23) at (4,4) {$[\Cyc{6,10}]$};
\node[scale=\myScale] (31) at (-2,6) {$[\Cyc{2,3}]$};
\node[scale=\myScale] (33) at (4,6) {$[\Cyc{2,5}]$};
\node[scale=\myScale] (34) at (6,6) {$[\Cyc{3,5}]$};
\node[scale=\myScale] (32) at (2,6) {$[\Cyc{30}]$};
\node[scale=\myScale] (41) at (0,8) {$[\Cyc{6}]$};
\node[scale=\myScale] (42) at (2,8) {$[\Cyc{10}]$};
\node[scale=\myScale] (43) at (4,8) {$[\Cyc{15}]$};
\node[scale=\myScale] (51) at (0,10) {$[\Cyc{2}]$};
\node[scale=\myScale] (52) at (2,10) {$[\Cyc{3}]$};
\node[scale=\myScale] (53) at (4,10) {$[\Cyc{5}]$};
\node[scale=\myScale] (60) at (2,12) {$[\Cyc{1}]$};
\path 
    (0)  edge (00)
    (00) edge (11)
    (00) edge (12)
    (00) edge (13)
    (11) edge (21)
    (11) edge (22)
    (12) edge (21)
    (12) edge (23)
    (13) edge (22)
    (13) edge (23)
    (21) edge (31)
    (21) edge (32)
    (22) edge (32)
    (22) edge (33)
    (23) edge (32)
    (23) edge (34)
    (31) edge (41)
    (32) edge (41)
    (32) edge (42)
    (32) edge (43)
    (33) edge (42)
    (34) edge (43)
    (41) edge (51)
    (41) edge (52)
    (42) edge (51)
    (42) edge (53)
    (43) edge (52)
    (43) edge (53)
    (51) edge (60)
    (52) edge (60)
    (53) edge (60)
    ;
\end{tikzpicture}
\hspace{1cm}
\begin{tikzpicture}[scale=0.525]
\def\myScale{0.4}

\node at (2,-2.2)  {};

\node[circle, draw,scale=\myScale] (0) at (2,-2)  {};
\node[circle, draw,scale=\myScale] (00) at (2,0)  {};
\node[circle, draw,scale=\myScale] (11) at (0,2)  {};
\node[circle, draw,scale=\myScale] (12) at (2,2) {};
\node[circle, draw,scale=\myScale] (13) at (4,2) {};
\node[circle, draw,scale=\myScale] (21) at (0,4) {};
\node[circle, draw,scale=\myScale] (22) at (2,4) {};
\node[circle, draw,scale=\myScale] (23) at (4,4) {};
\node[circle, fill,scale=\myScale] (31) at (-2,6) {};
\node[circle, fill,scale=\myScale] (33) at (4,6) {};
\node[circle, fill,scale=\myScale] (34) at (6,6) {};
\node[circle, draw,scale=\myScale] (32) at (2,6) {};
\node[circle, draw,scale=\myScale] (41) at (0,8) {};
\node[circle, draw,scale=\myScale] (42) at (2,8) {};
\node[circle, draw,scale=\myScale] (43) at (4,8) {};
\node[circle, draw,scale=\myScale] (51) at (0,10) {};
\node[circle, draw,scale=\myScale] (52) at (2,10) {};
\node[circle, draw,scale=\myScale] (53) at (4,10) {};
\node[circle, draw,scale=\myScale] (60) at (2,12) {};
\node[circle, draw,scale=\myScale] (70) at (2,14)  {};

\path 
    (0)  edge (00)
    (00) edge (11)
    (00) edge (12)
    (00) edge (13)
    (11) edge (21)
    (11) edge (22)
    (12) edge (21)
    (12) edge (23)
    (13) edge (22)
    (13) edge (23)
    (21) edge (31)
    (21) edge (32)
    (22) edge (32)
    (22) edge (33)
    (23) edge (32)
    (23) edge (34)
    (31) edge (41)
    (32) edge (41)
    (32) edge (42)
    (32) edge (43)
    (33) edge (42)
    (34) edge (43)
    (41) edge (51)
    (41) edge (52)
    (42) edge (51)
    (42) edge (53)
    (43) edge (52)
    (43) edge (53)
    (51) edge (60)
    (52) edge (60)
    (53) edge (60)
    (60) edge (70)
    ;
\end{tikzpicture}
\caption{The poset $\SDPoset$ restricted to disjoint unions of cycles using 2, 3, and 5 (left). The free distributive lattice on $3$ generators (right).}
    \label{fig:freedist3}
\end{figure}
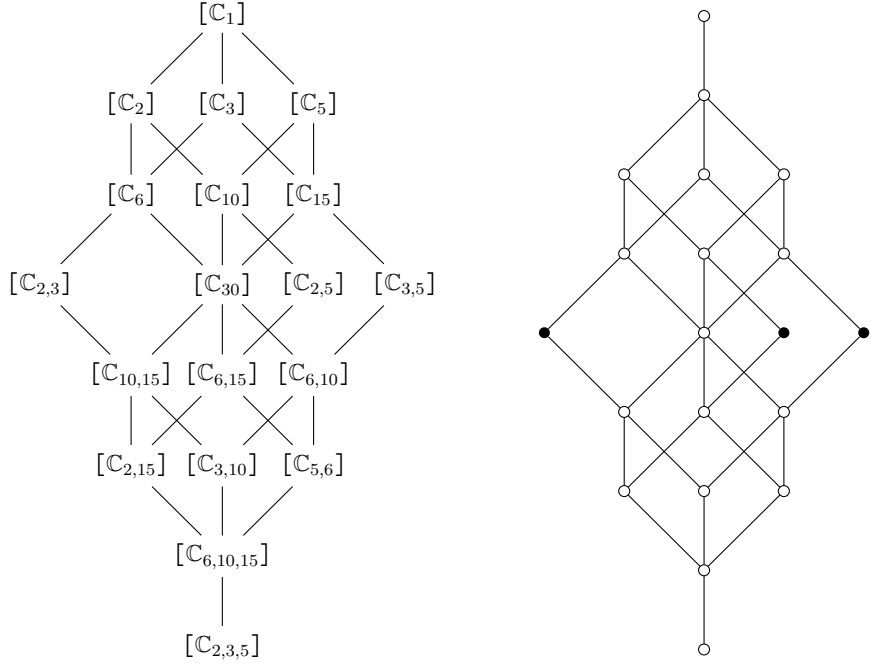
Observe that the poset in Figure~\ref{fig:freedist3} (left) is isomorphic to  the free distributive lattice on $3$ generators, $\mathcal F_D(3)$, after removing the top element. 
More generally, whenever we restrict $\SDPoset$ to disjoint unions of cycles using only a fixed finite set of $n$ primes, then the resulting poset is isomorphic to $\mathcal F_D(n)$ (again, after removing the top element).

Consider the power set $2^X$ of a countably infinite set $X$ ordered by inclusion. We denote by $\FD$ the poset of all downsets of $2^X$ ordered by inclusion.
Markowsky proved that $\FD$ is the free completely distributive lattice, i.e., complete and distributive over infinite meets and joins, on countably many generators~\cite{Markowsky}. 
The generating set consists of the principal downsets generated by $X\setminus\{x\}$ for $x\in X$. If we choose $X$ to be the set of all primes, then the following corollary is an immediate consequence of Corollaries~\ref{cor:characterizationMCLAndmCL} and~\ref{thm:classificationPoset}. First, for every finite set $C\subset\N^+$ define $\operatorname{Drop}(\Sigma_C)\coloneqq C$.
\begin{corollary}\label{thm:gorgeous}
The following holds
\begin{align*}
    \SDPoset&\hookrightarrow\MPCL\hookrightarrow\FD.
\end{align*}
One choice for the second embedding $\iota$ is to map an element of $\MPCL$, represented by a set of prime cyclic loop conditions $\Gamma$, to the downset of $\operatorname{Drop}(\Gamma)$.

\end{corollary}
Note that, since $\operatorname{Drop}(\Gamma)$ only contains finite sets it is necessary to take the downset to obtain an element of $\FD$ whose nonempty elements also contain infinite sets.
Furthermore, the image of $\iota$ is closed under finite meets and joins. Hence, the subposet $\iota(\MPCL)$ is also a sublattice of $\FD$.  Therefore, $\MPCL$ is a distributive lattice. Analogously, we obtain that $\SDPoset$ is a distributive lattice.

\begin{corollary}\label{cor:PSDIsLattice}
The poset of finite smooth digraphs ordered by pp-constructability is a complete and distributive lattice.
\end{corollary}
\begin{proof}
We have already argued that $\SDPoset$ is a distributive lattice. Adding a bottom element does not destroy distributivity. For completeness, consider an infinite set $\mathcal{C}\subseteq\SDPoset$. Note that every element in $\SDPoset$ has only finitely many elements above it. Hence, $\bigwedge\mathcal{C}=[\KThree]$. 
Since $\bigvee\mathcal{C}=\bigwedge\{\C\mid \C\geq \D\text{ for all }\D\in\mathcal C\}$ we have that $\mathcal C$ also has a supremum. 
\end{proof}
One could wonder whether $\SDPoset$ also distributes over infinite meets and joins. This is not the case as shown by this counterexample, provided (personal communication) by Friedrich Martin Schneider,
\[[\Cyc2]\vee\bigwedge_{\substack{p\text{ an odd}\\\text{prime}}}[\Cyc p]=[\Cyc2]\neq[\Cyc1]=\bigwedge_{\substack{p\text{ an odd}\\\text{prime}}} (\Cyc2\vee\Cyc p).\]

Here we summarize the results from
Corollaries~\ref{cor:characterizationMCLAndmCL} and~\ref{cor:mPclEqPPC} and Theorem~\ref{thm:classificationPoset} about the posets that have been studied in this article.
\begin{itemize}
    \item $\mPCL\simeq \PCPoset\simeq (\{P\mid P\text{ a finite nonempty set of primes}\},\supseteq)$
    \item $\mCL\simeq(\operatorname{Finitely Generated Downsets Of}(\mPCL),\subseteq)$.
    \item $\MCL=\MPCL\simeq(\operatorname{ Downsets Of}(\mPCL),\subseteq)$.
    \item $\SDPoset\simeq(\operatorname{Cofinite Downsets Of}(\mPCL),\subseteq)$.
\end{itemize}

Note that, $\mCL$ is isomorphic to a sublattice of $\MPCL$, hence it is a distributive lattice, as well. The meet and join of $\mCL$ can be described in the following way.

\begin{lemma}\label{lem:meetAndJoinInMCL}
Let $[\Sigma_C], [\Sigma_D]\in\mCL$. Then
\begin{enumerate}
    \item $[\Sigma_C]\wedge[\Sigma_D]=[\Sigma_{C\cup D}]$ and
    \item $[\Sigma_C]\vee[\Sigma_D]=[\Sigma_{C\cdot D}]=[\Sigma_{\C\mathbin\times \D}]$.
\end{enumerate} 
\end{lemma}
\begin{proof}
Let $\Sigma_E$ and $\Sigma_{E'}$ be cyclic loop conditions.
Assume without loss of generality that the sets $C$, $D$, $E$, and $E'$ contain only square-free numbers.

Clearly, $\C\to\C\cup\D$, $\D\to \C\cup\D$, $\C\mathbin\times\D\to\C$, and $\C\mathbin\times\D\to\D$. Suppose that $\Sigma_C\Rightarrow\Sigma_E$, $\Sigma_D\Rightarrow\Sigma_E$, 
$\Sigma_{E'}\Rightarrow\Sigma_{C}$, and $\Sigma_{E'}\Rightarrow\Sigma_{D}$. Then, by Theorem~\ref{thm:implicationSingleClc}, there are homomorphisms $f,g,f',$ and $g'$ as in Figure~\ref{fig:meetAndJoin}. 
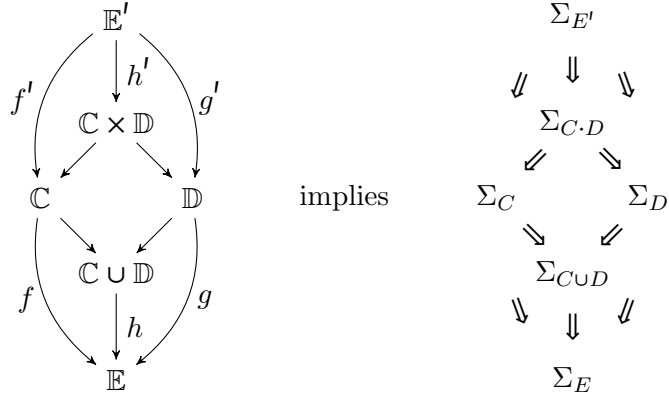
\begin{figure}
    \centering
    \begin{tikzpicture}
        \node (C) at (-1,0) {$\C$};
        \node (D) at (1,0) {$\D$};
        \node (CuD) at (0,-1) {$\C\cup\D$};
        \node (E) at (0,-2.414) {$\mathbb E$};
        \node (CxD) at (0,1) {$\C\mathbin\times\D$};
        \node (Ep) at (0,2.414) {$\mathbb E'$};
        
        \path[->,>=stealth']
            (C) edge (CuD)
            (D) edge (CuD)
            (CxD) edge (C)
            (CxD) edge (D)
            (Ep) edge node[right] {$h'$} (CxD)
            (Ep) edge[bend right] node[left] {$f'$} (C)
            (Ep) edge[bend left] node[right] {$g'$} (D)
            (CuD) edge node[right] {$h$}(E)
            (C) edge[bend right] node[left] {$f$}(E)
            (D) edge[bend left] node[right] {$g$}(E)
            ;
        \node at (3,0) {implies};
        \node (C) at (-1+6,0) {$\Sigma_C$};
        \node (D) at (1+6,0) {$\Sigma_D$};
        \node (CuD) at (0+6,-1) {$\Sigma_{C\cup D}$};
        \node (E) at (0+6,-2.414) {$\Sigma_E$};
        \node (CxD) at (0+6,1) {$\Sigma_{C\cdot D}$};
        \node (Ep) at (0+6,2.414) {$\Sigma_{E'}$};
        
        \node[rotate=-90] at (0+6,-1.707) {$\Rightarrow$};
        \node[rotate=-90] at (0+6,1.707) {$\Rightarrow$};
        \node[rotate=-70] at (-0.7+6,-1.5) {$\Rightarrow$};
        \node[rotate=-110] at (0.7+6,-1.5) {$\Rightarrow$};
        \node[rotate=-45] at (-0.5+6,-0.5) {$\Rightarrow$};
        \node[rotate=-135] at (0.5+6,-0.5) {$\Rightarrow$};
        \node[rotate=-110] at (-0.7+6,1.5) {$\Rightarrow$};
        \node[rotate=-70] at (0.7+6,1.5) {$\Rightarrow$};
        \node[rotate=-135] at (-0.5+6,0.5) {$\Rightarrow$};
        \node[rotate=-45] at (0.5+6,0.5) {$\Rightarrow$};
        
    \end{tikzpicture}
    \caption{Determining the meet and join of $[\Sigma_C]$ and $[\Sigma_D]$.}
    \label{fig:meetAndJoin}
\end{figure}
The maps 
\begin{align*}
    h\colon\C\cup\D&\to \mathbb E & h'\colon\mathbb E'&\to\C\mathbin\times\D\\
u&\mapsto
\begin{cases}
f(u)&\text{if }u\in\C\\
g(u)&\text{otherwise}
\end{cases} &
u&\mapsto (f(u),g(u))
\end{align*}
are homomorphism as well. Hence $[\Sigma_{C\cup D}]$ is the meet and $[\Sigma_{C\mathbin\times D}]$ is the join of $[\Sigma_C]$ and $[\Sigma_D]$. 
\end{proof}

Using Lemmata~\ref{lem:surjectivety} and~\ref{lem:pclCharacterization} we can describe the meet of $\SDPoset$ of disjoint unions of prime cycles.
\begin{corollary}
For any finite antichain $Q$ of $\PCPoset$ we have that
\[\bigwedge_{\mathbb P\in Q} [\mathbb P]=\left[\bigtimes_{\mathbb P\in Q} \mathbb P\right].\]
\end{corollary}

Note that this connection between $\times$ and $\wedge$ does not hold in general as seen in the following example.
\begin{example}
We have that $[\Cyc{30}\mathbin\times\Cyc{2,3}]=[\Cyc{30}]$ and, as one can see in Figure~\ref{fig:freedist3}, $[\Cyc{30}]\wedge[\Cyc{2,3}]=[\Cyc{10,15}]$. 
We have that $[\Cyc{10,15}\mathbin\times\Cyc{6,10}]=[\Cyc{10}]$ and, as one can see in Figure~\ref{fig:freedist3}, $[\Cyc{10,15}]\wedge[\Cyc{6,10}]=[\Cyc{3,10}]$.
\eoe
\end{example}

Observe that a prime cyclic loop condition $[\Sigma_P]\in\mCL$ corresponds to a downset of $\mPCL$, which is generated by a single element. Analogously, a disjoint union of prime cycles $[\mathbb P]\in\SDPoset$ corresponds to the complement of an upset of $\mPCL$, which is generated by a single element. 

\begin{corollary}\label{cor:joinAndMeetIrred}
The join irreducible elements of $\mCL\setminus\{[\Sigma_1]\}$ are exactly the elements that can be represented by a prime cyclic loop condition.

The meet irreducible elements of $\SDPoset\setminus\{[\Cyc 1]\}$ are exactly the elements that can be represented by a finite disjoint union of prime cycles.
\end{corollary}

Let $P$ be a finite set of primes and $\N^+_P$ be the set of all positive natural numbers whose prime decomposition only uses numbers from $P$. 
Observe that
the subposet of $\SDPoset$ consisting of $\{[\C]\mid C\subset \N_P^+ \text{ finite}\}$ and the subposet of $\mCL$ consisting of $\{[\Sigma_C]\mid C\subset \N_P^+  \text{ finite}\}$ are isomorphic. 
The map 
\[[\C]\mapsto[\{\Sigma_D\mid D\subset\N_P^+\text{ finite, }\C\models\Sigma_D\}]\]
is an isomorphism.
See for example Figure~\ref{fig:SDPvsmCL}. 
However, the posets $\mCL$ and $\SDPoset$ are not isomorphic since $\mCL$ has infinite ascending chains, e.g. $([\Sigma_{p_1\cdot\ldots\cdot p_n}])_{n\in\N}$, and $\SDPoset$ does not.

\begin{figure}
    \centering
    \begin{tikzpicture}[scale=0.6]
\def\myScale{0.8}

\node[scale=\myScale] (0) at (2,-2)  {$\boldsymbol{[\pmb{\C}_{2,3,5}]}$};
\node[scale=\myScale] (00) at (2,0)  {$[\Cyc{6,10,15}]$};
\node[scale=\myScale] (11) at (0,2)  {$[\Cyc{2,15}]$};
\node[scale=\myScale] (12) at (2,2) {$[\Cyc{3,10}]$};
\node[scale=\myScale] (13) at (4,2) {$[\Cyc{5,6}]$};
\node[scale=\myScale] (21) at (0,4) {$[\Cyc{10,15}]$};
\node[scale=\myScale] (22) at (2,4) {$[\Cyc{6,15}]$};
\node[scale=\myScale] (23) at (4,4) {$[\Cyc{6,10}]$};
\node[scale=\myScale] (31) at (-2,6) {$\boldsymbol{[\pmb{\C}_{2,3}]}$};
\node[scale=\myScale] (33) at (4,6) {$\boldsymbol{[\pmb{\C}_{2,5}]}$};
\node[scale=\myScale] (34) at (6,6) {$\boldsymbol{[\pmb{\C}_{3,5}]}$};
\node[scale=\myScale] (32) at (2,6) {$[\Cyc{30}]$};
\node[scale=\myScale] (41) at (0,8) {$[\Cyc{6}]$};
\node[scale=\myScale] (42) at (2,8) {$[\Cyc{10}]$};
\node[scale=\myScale] (43) at (4,8) {$[\Cyc{15}]$};
\node[scale=\myScale] (51) at (0,10) {$\boldsymbol{[\pmb{\C}_{2}]}$};
\node[scale=\myScale] (52) at (2,10) {$\boldsymbol{[\pmb{\C}_{3}]}$};
\node[scale=\myScale] (53) at (4,10) {$\boldsymbol{[\pmb{\C}_{5}]}$};
\node[scale=\myScale] (60) at (2,12) {$\boldsymbol{[\pmb{\C}_{1}]}$};
\path 
    (0)  edge (00)
    (00) edge (11)
    (00) edge (12)
    (00) edge (13)
    (11) edge (21)
    (11) edge (22)
    (12) edge (21)
    (12) edge (23)
    (13) edge (22)
    (13) edge (23)
    (21) edge (31)
    (21) edge (32)
    (22) edge (32)
    (22) edge (33)
    (23) edge (32)
    (23) edge (34)
    (31) edge (41)
    (32) edge (41)
    (32) edge (42)
    (32) edge (43)
    (33) edge (42)
    (34) edge (43)
    (41) edge (51)
    (41) edge (52)
    (42) edge (51)
    (42) edge (53)
    (43) edge (52)
    (43) edge (53)
    (51) edge (60)
    (52) edge (60)
    (53) edge (60)
    ;
\end{tikzpicture}
\hspace{1cm}
    \begin{tikzpicture}[scale=0.6]
\def\myScale{0.8}

\node[scale=\myScale] (0) at (2,-2)  {$\boldsymbol{[\Sigma_{1}]}$};
\node[scale=\myScale] (00) at (2,0)  {$\boldsymbol{[\Sigma_{2,3,5}]}$};
\node[scale=\myScale] (11) at (0,2)  {$\boldsymbol{[\Sigma_{3,5}]}$};
\node[scale=\myScale] (12) at (2,2) {$\boldsymbol{[\Sigma_{2,5}]}$};
\node[scale=\myScale] (13) at (4,2) {$\boldsymbol{[\Sigma_{2,3}]}$};
\node[scale=\myScale] (21) at (0,4) {$[\Sigma_{5,6}]$};
\node[scale=\myScale] (22) at (2,4) {$[\Sigma_{3,10}]$};
\node[scale=\myScale] (23) at (4,4) {$[\Sigma_{2,15}]$};
\node[scale=\myScale] (31) at (-2,6) {$\boldsymbol{[\Sigma_{5}]}$};
\node[scale=\myScale] (33) at (4,6) {$\boldsymbol{[\Sigma_{3}]}$};
\node[scale=\myScale] (34) at (6,6) {$\boldsymbol{[\Sigma_{2}]}$};
\node[scale=\myScale] (32) at (2,6) {$[\Sigma_{6,10,15}]$};
\node[scale=\myScale] (41) at (0,8) {$[\Sigma_{10,15}]$};
\node[scale=\myScale] (42) at (2,8) {$[\Sigma_{6,15}]$};
\node[scale=\myScale] (43) at (4,8) {$[\Sigma_{6,10}]$};
\node[scale=\myScale] (51) at (0,10) {$[\Sigma_{15}]$};
\node[scale=\myScale] (52) at (2,10) {$[\Sigma_{10}]$};
\node[scale=\myScale] (53) at (4,10) {$[\Sigma_{6}]$};
\node[scale=\myScale] (60) at (2,12) {$[\Sigma_{30}]$};
\path 
    (0)  edge (00)
    (00) edge (11)
    (00) edge (12)
    (00) edge (13)
    (11) edge (21)
    (11) edge (22)
    (12) edge (21)
    (12) edge (23)
    (13) edge (22)
    (13) edge (23)
    (21) edge (31)
    (21) edge (32)
    (22) edge (32)
    (22) edge (33)
    (23) edge (32)
    (23) edge (34)
    (31) edge (41)
    (32) edge (41)
    (32) edge (42)
    (32) edge (43)
    (33) edge (42)
    (34) edge (43)
    (41) edge (51)
    (41) edge (52)
    (42) edge (51)
    (42) edge (53)
    (43) edge (52)
    (43) edge (53)
    (51) edge (60)
    (52) edge (60)
    (53) edge (60)
    ;
\end{tikzpicture}
\caption{The poset $\SDPoset$ restricted to disjoint unions of cycles using 2, 3, and 5; meet irreducible elements are bold (left). The poset $\mCL$ restricted to cyclic loop conditions using 2, 3, and 5; join irreducible elements are bold (right).}
    \label{fig:SDPvsmCL}
\end{figure}
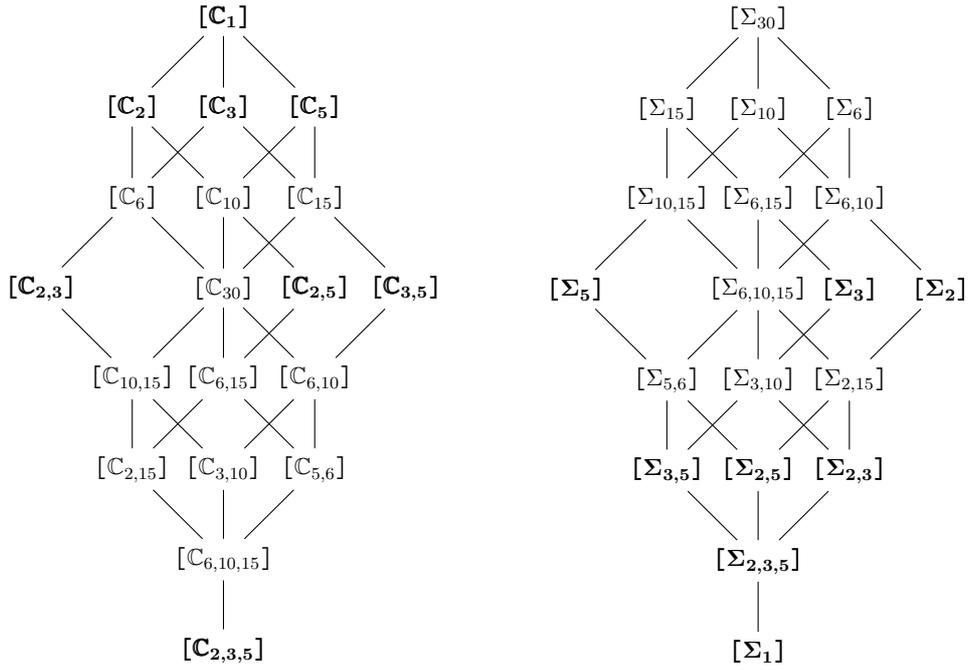

\section{Conclusion}
In the present article we described the poset $\SDPoset$, i.e., the subposet of $\PPPoset$ where every element is a pp-constructability type of some finite disjoint union of cycles. From the provided description it follows that $\SDPoset$ is a distributive lattice and that it contains infinite descending chains and infinite antichains. Some of these properties are inherited by $\PPPoset$.
For instance, it follows that $\PPPoset$ contains  infinite antichains and infinite descending chains; the latter was already known from the description of $\mathfrak{P}_{\operatorname{Boole}}$~\cite{albert}. 
\begin{question}
Is $\PPPoset$ a lattice?
\end{question}

Another direction for future work is to drop the smoothness assumption and to characterize the subposet $\mathfrak{P}_\text{D}$ of $\PPPoset$ consisting of the pp-contructability types of finite digraphs. For every element $[\mathbb G]$ of $\mathfrak{P}_\text{D}$ different from $[\tikz{
\node at (-30:0.2) [circle, fill, scale=0.3] (0) {};
}]$ we have that   $[\mathbb G]\leq [\tikz{
\node at (0.2,0) [circle, fill, scale=0.3] (1) {};
\node at (-0.15,0) [circle, fill, scale=0.3] (2) {};
\path[->,>=stealth'] 
    (2) edge (1);
}]$.
\begin{conjecture}
In  $\mathfrak{P}_\text{D}$ all elements of the form $[\Cyc{p}]$, where $p$ is a prime number, are
covered by $[\tikz{
\node at (0.2,0) [circle, fill, scale=0.3] (1) {};
\node at (-0.15,0) [circle, fill, scale=0.3] (2) {};
\path[->,>=stealth'] 
    (2) edge (1);
}]$, i.e., there are no $p$ and $[\mathbb G]\in \mathfrak{P}_\text{D}$ with $[\tikz{
\node at (0.2,0) [circle, fill, scale=0.3] (1) {};
\node at (-0.15,0) [circle, fill, scale=0.3] (2) {};
\path[->,>=stealth'] 
    (2) edge (1);
}]>[\mathbb G]>[\Cyc p]$. 
\end{conjecture}

From Theorem~\ref{thm:implicationSingleClc} we concluded that given  two finite sets $C,D\subset\N^+$ it is decidable whether $\Sigma_C\Rightarrow\Sigma_D$. 
\begin{question}
Is the following problem decidable:
\begin{align*}
    &\text{Input: two height 1 conditions $\Gamma$ and $\Sigma$}\\
    &\text{Output: Does $\Gamma\Rightarrow\Sigma$ hold?}
\end{align*}
\end{question}

\textbf{Acknowledgement.} 
The authors thank Jakub Opr\v{s}al for being the pioneer who tormented
himself with the first draft of this article. We would also like to
thank the anonymous referee for thoroughly digging through every proof.
Their comments shaped the article and greatly simplified the life of any future
reader.

\end{document}